\documentclass[a4paper]{article}

\title{Numerical computation of the half Laplacian by means of a fast convolution algorithm\thanks{%
		Work partially supported by the research group grant IT1615-22 funded by the Basque Government, by the project PGC2018-094522-B-I00 funded by MICINN,
		and by the project PID2021-126813NB-I00 funded by MCIN/AEI/10.13039/501100011033 and by ``ERDF A way of making Europe''. Ivan Girona was also partially supported by the MICINN PhD grant PRE2019-090478.
}}

\author{Carlota M. Cuesta\footnotemark[2] \and Francisco de la Hoz\footnotemark[2] \and Ivan Girona\footnotemark[2]}

\usepackage{amsfonts,amssymb,amsmath,amsthm}
\usepackage{verbatim}
\usepackage{graphicx}

\newtheorem{theorem}{Theorem}[section]
\newtheorem{corollary}[theorem]{Corollary}
\newtheorem{lemma}[theorem]{Lemma}

\makeatletter
\newcommand{\biggg}{\bBigg@\thr@@}
\newcommand{\Biggg}{\bBigg@{3.5}}

\makeatother

\usepackage[T1]{fontenc}

\usepackage{listings}

\DeclareMathOperator{\sgn}{sgn}

\DeclareMathOperator{\atanh}{\arg\tanh}
\DeclareMathOperator{\acot}{arccot}
\DeclareMathOperator{\erf}{erf}
\DeclareMathOperator{\Res}{Res}

\begin{document}

	\maketitle
	
	\renewcommand{\thefootnote}{\fnsymbol{footnote}}
	
	\footnotetext[2]{Department of Mathematics, Faculty of Science and Technology, University of the Basque Country UPV/EHU, Barrio Sarriena S/N, 48980 Leioa, Spain}
	
	\begin{abstract} 
		In this paper, we develop a fast and accurate pseudospectral method to approximate numerically the half Laplacian $(-\Delta)^{1/2}$ of a function on $\mathbb{R}$, which is equivalent to the Hilbert transform of the derivative of the function.
		
		The main ideas are as follows. Given a twice continuously differentiable bounded function $u\in\mathcal C_b^2(\mathbb{R})$, we apply the change of variable $x=L\cot(s)$, with $L>0$ and $s\in[0,\pi]$, which maps $\mathbb{R}$ into $[0,\pi]$, and denote $(-\Delta)_s^{1/2}u(x(s)) \equiv (-\Delta)^{1/2}u(x)$. Therefore, by performing a Fourier series expansion of $u(x(s))$, the problem is reduced to computing  $(-\Delta)_s^{1/2}e^{iks} \equiv (-\Delta)^{1/2}[(x + i)^k/(1+x^2)^{k/2}]$. On a previous work, we considered the case with $k$ even for the more general power $\alpha/2$, with $\alpha\in(0,2)$, so here we focus on the case with $k$ odd. More precisely, we express $(-\Delta)_s^{1/2}e^{iks}$ for $k$ odd in terms of the Gaussian hypergeometric function ${}_2F_1$, and also as a well-conditioned finite sum. Then, we use a fast convolution result, that enable us to compute very efficiently $\sum_{l = 0}^Ma_l(-\Delta)_s^{1/2}e^{i(2l+1)s}$, for extremely large values of $M$. This enables us to approximate $(-\Delta)_s^{1/2}u(x(s))$ in a fast and accurate way, especially when $u(x(s))$ is not periodic of period $\pi$. As an application, we simulate a fractional Fisher's equation having front solutions whose speed grows exponentially.
	\end{abstract}
	
	
	
	\section{Introduction}
	
	In this paper, we develop a fast and spectrally accurate pseudospectral method to approximate numerically the operator $(-\Delta)^{1/2}$ on $\mathbb R$, known as the half Laplacian or the half fractional Laplacian. This operator appears in the modeling of relevant physical phenomena, such as the Peierls-Nabarro model describing the motion of dislocations in crystals (see, e.g., \cite{Monneau2012}), or the Benjamin-Ono equation \cite{Benjamin1975,Ono1975} in hydrodynamics. Among the different ways of defining the fractional Laplacian (see, e.g., \cite{kwasnicki}, where $-(-\Delta)^{\alpha/2}$ is considered, for $\alpha\in(0,2)$), we employ the definition in  \cite{cayamacuestadelahoz2020,cayamacuestadelahoz2021, cayamacuestadelahozgarciacervera2022} (in this paper, all the integrals that are not absolutely convergent must be understood in the principal value sense):
	\begin{equation}
		\label{e:fraclapl}
		(-\Delta)^{1/2}u(x) = \frac{1}{\pi}\int_{-\infty}^\infty\frac{u(x)-u(x+y)}{y^2}dy.
	\end{equation}
	On the other hand, whenever $u$ is a twice continuously differentiable bounded function, i.e., $u\in\mathcal C_b^2(\mathbb{R})$, we can express \eqref{e:fraclapl} as (see \cite{cayamacuestadelahoz2021}):
	\begin{equation}
		\label{e:fraclaplmathematica}
		(-\Delta)^{1/2}u(x) = \frac{1}{\pi}\int_{0}^{\infty}\frac{u_x(x-y) - u_x(x+y)}{y}dy ,
	\end{equation}
	or, equivalently,
	\begin{equation}
		\label{e:fraclapl2}
		(-\Delta)^{1/2}u(x) = \frac{1}{\pi}\int_{-\infty}^\infty \frac{u_{x}(y)}{x-y}dy,
	\end{equation}
	i.e., the fractional Laplacian is the Hilbert transform of the derivative of $u(x)$, i.e.,
	\begin{equation*}
		(-\Delta)^{1/2}u(x) \equiv \mathcal H(u_x(x)).
	\end{equation*}
	Although there are references specifically aimed at the numerical computation of the Hilbert transform (see, e.g., \cite{Olver2011} and \cite{weideman1995}), in this paper we have rather opted to study $(-\Delta)^{1/2}u(x)$ taking \cite{cayamacuestadelahoz2021} as our starting point. More precisely, after applying the mapping $x = L\cot(s)$, with $L > 0$, which maps the real line $\mathbb R$ into the finite interval $s\in[0, \pi]$ and hence avoids truncation of the domain, \eqref{e:fraclapl2} becomes
	\begin{equation}
		\label{e:fraclaps}
		(-\Delta)_s^{1/2}U(s) = -\frac{1}{L\pi}\int_{0}^\pi \frac{U_s(\eta)}{\cot(s) - \cot(\eta)}d\eta = \frac{\sin(s)}{L\pi}\int_{0}^\pi \frac{\sin(\eta)U_s(\eta)}{\sin(s - \eta)}d\eta,
	\end{equation}
	where
	\begin{equation*}
		U(s) \equiv u(x(s)), \qquad (-\Delta)_s^{1/2}U(s) \equiv (-\Delta)^{1/2}u(x).
	\end{equation*}
	Moreover, along this paper, we usually denote the derivatives by means of subscripts, so $u_{x}(x) \equiv - (\sin^{2}(s)/L)U_{s}(s)$ and $dx = -L\sin^{-2}(s)ds$.
	
	It is immediate to check that $(-\Delta)_s^{1/2}U(s)$ is periodic of period $\pi$ in $s$, which suggests expanding $U(s) \equiv u(x(s))$ in terms of $e^{iks}$, with $k\in\mathbb Z$. This was done in \cite{cayamacuestadelahoz2021}, getting the following result:
	\begin{equation}
		\label{e:tha1}
		(-\Delta)_s^{1/2}e^{iks} =
		\left\lbrace
		\begin{aligned}
			& \frac{|k|\sin^2(s)}{L}e^{iks}, & \text{$k$ even,}
			\\
			& {\frac{ik}{L\pi}}\left[\frac{2}{4 - k^2} - \sum_{n=-\infty}^\infty\frac{4\sgn(n)e^{i2ns}}{(2n - k)(4 - (2n - k)^2)}\right], & \text{$k$ odd,}
		\end{aligned}
		\right.
	\end{equation}
	where $(-\Delta)_s^{1/2}e^{iks} \equiv (-\Delta)^{1/2}e^{ik\acot(x/L)}$, and the multivalued real-valued function $\acot(x)$ is considered to take values in $(0,\pi)$. Hence, when $L = 1$, if $U(s) = e^{iks}$, we have the following equivalences:
	\begin{equation*}
		U(s) = e^{iks} \Longleftrightarrow u(x) = \left(\frac{x + i}{x - i}\right)^{k/2} = \left(\frac{x - i}{x + i}\right)^{-k/2} = \frac{(x + i)^k}{(1+x^2)^{k/2}} = \frac{(x - i)^{-k}}{(1+x^2)^{-k/2}}.
	\end{equation*}
	Note that, in general, from a numerical point of view, it is advisable to work with $s$ rather than with $x$. Indeed, $e^{iks}$, for $k$ integer, is unequivocally defined, whereas a direct numerical evaluation of  $((x + i) / (x - i))^{k/2}$ may yield unexpected results when $k$ is odd. Therefore, in the latter case, it is preferable to operate with $(x+i)^k/(1+x^2)^{k/2}$ (especially if $k$ is positive) or with $(x-i)^{-k}/(1+x^2)^{-k/2}$ (especially if $k$ is negative). Throughout this paper, we will always consider the identity $e^{iks}\equiv(x+i)^k/(1+x^2)^{k/2}$. Then,
	$$
	\frac{d}{dx}\frac{(x + i)^k}{(1+x^2)^{k/2}} = \frac{-i2k}{1+x^2}\frac{(x + i)^k}{(1+x^2)^{k/2}} = \frac{-i2k(x + i)^k}{(1+x^2)^{(1+k)/2}},
	$$
	and $(-\Delta)_s^{1/2}e^{iks}$ denotes
	\begin{align*}
		(-\Delta)^{1/2}_se^{iks} \equiv (-\Delta)^{1/2} \frac{(x + i)^k}{(1+x^2)^{k/2}} \equiv \mathcal H\left(\frac{-i2k(x + i)^k}{(1+x^2)^{(1+k)/2}}\right).
	\end{align*}
	Therefore, when $L = 1$, \eqref{e:tha1} can be expressed in terms of $x$ as follows:
	\begin{align}
		\label{e:tha1x}
		& (-\Delta)^{1/2} \frac{(x + i)^k}{(1+x^2)^{k/2}}
		\cr
		& \qquad = \left\lbrace
		\begin{aligned}
			& \frac{|k|(x + i)^k}{(1+x^2)^{k/2+1}}, & \text{$k$ even,}
			\\
			& \frac{ik}{\pi}\left[\frac{2}{4-k^2} - \sum_{n=-\infty}^\infty\frac{4\sgn(n)(x + i)^{2n}}{(2n-k)(4-(2n-k)^2)(1+x^2)^{n}}\right], & \text{$k$ odd,}
		\end{aligned}
		\right.
	\end{align}
	where we have used that $\sin^2(s) \equiv (1+x^2)^{-1}$. 
	
	As can be seen, \eqref{e:tha1} and \eqref{e:tha1x} pose no difficulty in the case with $k$ even. In this regard, note that the so-called complex Higgins functions $\lambda_k(x)$ (see, e.g., \cite{boyd1990}), which are the equivalent in terms of $x$ of $e^{i2ks}$, i.e.,
	$$
	\lambda_k(x)\equiv \left(\frac{x + i}{x - i}\right)^{k} \equiv e^{i2ks},
	$$
	form a complete orthogonal system in $\mathcal L^2(\mathbb R; 1/(\pi(1+x^2)))$. Therefore, bearing in mind that $1+x^2 = (1 + ix)(1 - ix)$, this suggests considering the so-called Christov functions $\mu_k(x)$ (see, e.g., \cite{boyd1990}):
	$$
	\mu_k(x)\equiv \frac{1}{ix+1}\left(\frac{x + i}{x - i}\right)^{k} = \frac{(ix - 1)^k}{(ix + 1)^{k+1}},
	$$
	which form a complete orthogonal system in $\mathcal L^2(\mathbb R)$ and are eigenfunctions of the Hilbert transform (see, e.g., \cite{weideman1995}); here, we must mention the closely related Malmquist-Takenaka functions \cite{Malmquist1926,Takenaka1925}, which differ from the Christov functions in a scaling and a constant factor. Note also that there are other orthogonal systems in $\mathcal L^2(\mathbb R)$ (see, e.g., \cite{IserlesWebb2018}, where the authors explore orthogonal systems in $\mathcal L^2(\mathbb R)$ which give rise to a real skewsymmetric, tridiagonal, irreducible differentiation matrix; and they conclude that the only such orthogonal system consisting of a polynomial sequence multiplied by a weight function is the Hermite functions).
	
	Given a function $u(x)$ that is of Schwartz class or analytic at infinity, the coefficients of its expansion in terms of the complex Higgins functions or the Christov functions decrease exponentially fast (see, e.g., \cite{boyd1990}). Observe also that $(x + i) / (x - i)$ maps the real line into the unit circle, so these expansions are reduced to Fourier series, by considering, e.g., $u(L\cot(s))$, as we do in this paper; in this regard, we refer to \cite{Trogdon2016}, where there is a complete study on the convergence of the interpolation of functions on the real axis with rational functions, by using a map from $\mathbb R$ to the unit circle.
	
	On the other hand, the computation of $(-\Delta)_s^{1/2}e^{iks}$, for $k$ odd, is much more involved than the $k$ even case, and will hence constitute one of the central aspects of this paper. Note that, when $k$ is odd, the main caveat of \eqref{e:tha1} is the presence of the infinite sum. Indeed, in the numerical implementation of \eqref{e:tha1} in \cite{cayamacuestadelahoz2021}, even if there was no truncation of the domain, it was still necessary to truncate the limits in the summation sign, to approximate numerically $(-\Delta)_s^{1/2}e^{iks}$ for $k$ odd. On the other hand, it was possible to consider simultaneously a large amount of odd values of $k$, by using a matrix representation, and, although the method was efficient even for moderately large amounts of frequencies and points, the size of the resulting dense matrix prevented considering extremely large problems (the largest value of $N$ in \cite{cayamacuestadelahoz2021} was $N = 8192$).
	
	The paper is organized as follows. In Section \ref{s:finite}, taking \eqref{e:tha1} as our starting point, we express, $(-\Delta)_s^{1/2}e^{iks}$ as a finite sum, for $k$ an odd integer. This constitutes one of the crucial points of this paper. Moreover, we show that $(-\Delta)_s^{1/2}e^{iks}$ can be written in a compact way by means of the Gaussian hypergeometric function ${}_2F_1$, which can also be reformulated as a finite sum. At this point, let us mention that, in \cite{cayamacuestadelahoz2020}, the general fractional Laplacian operator $(-\Delta)^{\alpha/2}$ applied to the complex Higgins functions $\lambda_k(x)$ (which is equivalent $(-\Delta)_s^{\alpha/2}e^{iks}$, for $k$ even) and to the Christov functions $\mu_k(x)$ was also expressed in terms of  ${}_2F_1$, for all $\alpha\in(0,2)$. This was possible, because $\lambda_k(x)$ and $\mu_k(x)$ are rational functions, and hence it is possible to apply a complex contour integration technique; however, when $k$ is odd, the denominator in $e^{iks}\equiv(x+i)^k/(1+x^2)^{k/2}$ is no longer a polynomial, which makes the analytical computation of $(-\Delta)_s^{1/2}e^{iks}$ much harder. In any case, the numerical computation of the instances of ${}_2F_1$ appearing in \cite{cayamacuestadelahoz2020} was challenging and required the use of multiple precision.
	
	Section \ref{s:convolution} is devoted to the other crucial point of this paper, namely the use of a fast convolution algorithm \cite{cayamacuestadelahozgarciacervera2022} to compute
	\begin{equation}
		\label{e:lincombfraclap}
		\sum_{l = 0}^Ma_l(-\Delta)_s^{1/2}e^{i(2l+1)s},
	\end{equation}
	for extremely large values of $M$ (in this paper, we consider values of $M$ of the order of $\mathcal O(10^7)$, but we have found that values of $M$ of the order of, e.g., $\mathcal O(10^8)$ also execute without problems in our computer). Note that the idea of a fast convolution method has been recently used in the context of the fractional Laplacian in \cite{cayamacuestadelahozgarciacervera2022}, for $\alpha\in(0,1)\cup(1,2)$, but in combination with a second-order modification of the midpoint rule, which yields an efficient method, although it is necessary to increase the amount of points, in order to achieve high accuracy. However, in this paper, we are able to compute \eqref{e:lincombfraclap} exactly, except for infinitesimally small rounding errors.
	
	The fast computation of \eqref{e:lincombfraclap} enables us to implement an efficient and accurate pseudospectral method to approximate numerically $(-\Delta)_s^{1/2}U(s)$, which is done in Section \ref{s:implementation}. We consider first the case with $U(s)$ periodic of period $\pi$, and afterward, the case with $U(s)$ not periodic of period $\pi$, where we put the main focus. In order to clarify the implementation details, we offer the actual Matlab \cite{matlab} codes.
	
	In Section \ref{s:experiments}, we test the proposed numerical method with a number of functions. In principle, an even extension of $U(s)$ at $s = \pi$ can always be used, to have at least a globally continuous function, i.e., $U(s)\in\mathcal C^0(\mathbb R)$, but we also consider briefly a trigonometric extension on $s\in[\pi,2\pi]$. Furthermore, by finding a function $v(x)$ whose half Laplacian is explicitly known, and such that $W(s) = U(s) - V(s)$, is periodic of period $\pi$, where $V(s)\equiv v(x(s))$, we study how to reduce the nonperiodic case to the periodic one.
	
	Finally, in Section \ref{s:fisher}, as a practical application, we simulate the following fractional Fisher's equation:
	\begin{equation}
		\label{e:fisher}
		u_t = -(-\Delta)^{1/2}u + f(u), \quad x\in\mathbb R,\ t \ge 0,
	\end{equation}
	for the so-called monostable nonlinearity $f(u) = u(1 - u)$. In the case of classical diffusion (i.e., with $\Delta u$ instead of $-(-\Delta)^{1/2}u$), this is a paradigm equation for pattern forming systems and for reaction-diffusion systems in general (see, e.g., the classical references \cite{Danilov1995,Kolmogorov1937}). On the other hand, a more general version of \eqref{e:fisher} has been proposed as a reaction-diffusion system with anomalous diffusion (see, e.g., \cite{MancinelliVergniVulpiani2002}, where $(-\Delta)^{\alpha/2}$, with $\alpha\in(0,2)$, is considered). Some fundamental analytical results can be found in, e.g., \cite{CabreSire2014}, with a more general non-local operator and in multiple dimensions. 
	
	The simulation of \eqref{e:fisher} is challenging, because it has front propagating solutions that, as for the monostable case, travel in one direction with a wave speed that increases exponentially in time (see \cite{CabreRoquejoffre2013,Engler2010}). Such solutions of \eqref{e:fisher} were simulated in \cite{cayamacuestadelahoz2021}, in the general case containing $(-\Delta)^{\alpha/2}$, with $\alpha\in(0,2)$. More precisely, the initial data corresponding to $\alpha = 1$ was
	\begin{equation}
		\label{e:u0}
		u_0(x) = \frac{1}{2} - \frac{x}{2\sqrt{1 + x^2}},
	\end{equation}
	which satisfies $u_0(x)\to 1$, as $x\to -\infty$, and $u_0(x)\to0$, as $x\to\infty$; in fact, $u_0(x)\sim x^{-2}/4$, as $x\to\infty$. This $u_0(x)$ gives rise to a front solution that travels to the right, invading the unstable state $u = 0$ soon after initiating the evolution, whereas $u(x, t)\to 1$, as $x\to-\infty$, for all $t > 0$; observe that $u=1$ is a stable constant state, that overtakes, in the form of a front, the unstable one $u=0$. In the current case, the analytical results \cite{cayamacuestadelahoz2021} predict that the front travels with a speed $c\sim e^t$ for $t$ large enough. Numerically, we can check this behavior, as follows: if, for a given $t$, $x_{0.5}(t)$ denotes the value of $x$, such that $u(x, t) = 0.5$, then we expect that $x_{0.5}'(t) \sim e^t$, as $t\to \infty$ (see \cite{CabreRoquejoffre2013}).
	
	Such analytical results were confirmed in \cite{cayamacuestadelahoz2021} in the general case. However, unlike, in \cite{cayamacuestadelahoz2021}, where only $N = 1024$ was considered, we are now able to consider much larger values of $N$, namely $N = 2^{20} = 1048576$, which would be completely unfeasible by considering the matrix-based numerical method described in \cite{cayamacuestadelahoz2021}. Indeed, being able to work with extremely large values of $N$ is particularly desirable, because it allows us to capture the speed of the traveling front with great accuracy.
	
	Additionally, in this paper, we have taken another initial data:
	\begin{equation}
		\label{e:u0b}
		u_0(x) = 10^{-4}\left(\frac{1}{2} - \frac{x}{2\sqrt{1 + x^2}}\right),
	\end{equation}
	i.e., $u_0(x)$ in \eqref{e:u0} multiplied by $10^{-4}$, which is now a non-symmetric perturbation of the unstable state $u=0$. In particular, $u_0(x)\to10^{-4}$, as $x\to-\infty$, and $u_0(x)\to0$, as $x\to\infty$. This initial data is particularly interesting, because, as with $u_0(x)$ in \eqref{e:u0}, it gives rise again to a front solution that travels to the right, and such that $x_{0.5}'(t) \sim e^t$, but, when $x\to-\infty$, the solution quickly tends to the stable state $u = 1$, as expected.  
	
	Unlike in the examples in Section \ref{s:experiments}, no analytic expression of $u(x, t)$ is known, except for $t = 0$. In fact, when taking $u_0(x)$ in \eqref{e:u0}, we can assume $\lim_{x\to\infty}u(x, t) = 0$, $\lim_{x\to-\infty}u(x, t) = 1$, for all $t > 0$; but when taking $u_0(x)$ in \eqref{e:u0b}, at least for $0 < t \ll 1$, we cannot assume that $\lim_{x\to-\infty}u(x, t) = 1$. Therefore, in our opinion, the safest (and simplest) option here is to consider an even extension of $U(s, t)$ at $s = \pi$, which, produces very good results. In fact, in evolution problems like \eqref{e:fisher}, considering another type of extension, or trying to find a function $V(s, t)$ such that $W(s, t) = U(s, t) - V(s, t)$ is periodic of period $\pi$ and regular enough, would require a careful analytical study of $u(x, t)$ and of its spatial derivatives, which would make the numerical implementation much more complex.
	
	The paper is concluded by an appendix showing how to compute analytically $(-\Delta)^{1/2}(1+x^4)^{-1}$ by means of complex variable techniques.
	
	All the simulations have been run in an Apple MacBook Pro (13-inch, 2020, 2.3 GHz Quad-Core Intel Core i7, 32 GB).
	
	\section{Expressing $(-\Delta)_s^{1/2}e^{iks}$ for $k$ odd as a finite sum} \label{s:finite}
	We start this section by recalling some well-known concepts. Given $z\in\mathbb{C}$, the Pochhammer symbol is defined as
	$$
	(z)_n =
	\begin{cases}
		z(z+1)\ldots(z+n-1), & n\in\mathbb N,
		\\
		1, & n = 0.
	\end{cases}
	$$
	Note that, if $z$ is zero or a negative integer, and $n > |z|$, then $(z)_n = 0$. On the other hand, whenever $z$ is not zero nor a negative integer, an equivalent definition is
	\begin{equation}
		\label{relposchgamma}
		(z)_n = \frac{\Gamma(z+n)}{\Gamma(z)},
	\end{equation}
	and, in particular, whenever $z\in\mathbb N$,
	$$
	(z)_n = \frac{(z+n-1)!}{(z-1)!}.
	$$
	The Pochhammer symbol is needed to define the Gaussian hypergeometric function ${}_2F_1$:
	\begin{equation}
		\label{e:2F1}
		{}_{2}F_1(a,b;c;z)=\sum_{n=0}^{\infty}\frac{(a)_n(b)_n}{(c)_n}\frac{z^n}{n!},
	\end{equation}
	for $a,b,c,z\in\mathbb{C}$. This series (see, e.g., \cite[p. 5]{Slater1966}) is absolutely convergent when $|z| < 1$, and divergent when $|z|>1$. When $z = 1$, the series converges when $\Re(c - a - b)> 0$, and diverges when $\Re(c - a - b) \le 0$. When $|z| = 1$, but $z\not=1$, which is the case we are interested in, the series is absolutely convergent when $\Re(c-a-b)> 0$, convergent but not absolutely convergent when $-1 < \Re(c-a-b)\le0$, and divergent when $\Re(c - a - b) < -1$. Finally, when $|z| = 1$, $z\not=1$, and $\Re(c-a-b) = -1$, the series is convergent but not absolutely convergent if $\Re(a + b) > \Re(ab)$, and divergent, if $\Re(a + b) \le \Re(ab)$. For more details on the Gaussian hypergeometric functions, see, e.g., \cite{Slater1966}.
	
	We also need the following auxiliary lemma.
	\begin{lemma}\label{lemma:atanh}
		Let $s\in(0,\pi)$. Then
		\begin{equation}
			\label{e:atanh}
			\atanh(e^{\pm is})=\sum_{n=0}^{\infty}\frac{e^{\pm i(2n+1)s}}{2n+1}=\frac{1}{2}\ln\left(\cot\left(\dfrac{s}{2}\right)\right)\pm i\frac{\pi}{4}.
		\end{equation}
	\end{lemma}
	
	\begin{proof} We compute the series expansion of $\atanh(z)$,
		$$
		(\atanh)'(z) = \frac{1}{1 - z^2} = \sum_{n = 0}^{\infty}z^{2n} \Longrightarrow \atanh(z) = \sum_{n=0}^{\infty}\frac{z^{2n+1}}{2n+1},
		$$
		which is absolutely convergent when $|z| < 1$. Moreover, by using, e.g., the well-known Dirichlet criterion, it follows that it is also convergent when $|z| = 1$, except for the cases $z = 1$ and $z = -1$. Therefore, evaluating it at $z = e^{\pm is}$, with $s\in(0,\pi)$,
		\begin{equation}
			\label{e:atanh1}
			\atanh(e^{\pm is}) = \sum_{n=0}^{\infty}\frac{e^{\pm i(2n+1)s}}{2n+1}.
		\end{equation}
		On the other hand,
		$$
		\atanh(z) = y \Longleftrightarrow z = \tanh(y) = \frac{e^y - e^{-y}}{e^y + e^{-y}},
		$$
		so, solving for $y$,
		$$
		y = \atanh(z) = \frac12\ln\left(\frac{1 + z}{1 - z}\right).
		$$
		Hence, evaluating at $z = e^{\pm is}$, for $s\in(0,\pi)$,
		\begin{equation}
			\label{e:atanh2}
			\atanh(e^{\pm is}) = \frac12\ln\left(\frac{1 + e^{\pm is}}{1 - e^{\pm is}}\right) = \frac{1}{2}\ln\left(\cot\left(\dfrac{s}{2}\right)\right)\pm i\frac{\pi}{4},
		\end{equation}
		where we have used that, when $s\in(0,\pi)$,
		\begin{align*}
			\frac{1 + e^{\pm is}}{1 - e^{\pm is}} = \frac{e^{\mp is/2} + e^{\pm is/2}}{e^{\mp is/2} - e^{\pm is/2}} = 0\pm i\cot\left(\dfrac{s}{2}\right) \Longrightarrow
			\left\{\begin{aligned}
				& \left|\frac{1 + e^{\pm is}}{1 - e^{\pm is}}\right|  = \cot\left(\dfrac{s}{2}\right),
				\cr
				& \arg\left(\frac{1 + e^{\pm is}}{1 - e^{\pm is}}\right) = \pm\frac\pi2.
			\end{aligned}\right.
		\end{align*}
		Putting \eqref{e:atanh1} and \eqref{e:atanh2} together, \eqref{e:atanh} follows.
	\end{proof}
	
	Lemma \ref{lemma:atanh} is used to prove the following result.
	\begin{lemma} Let $k$ be an odd integer, and $s\in(0,\pi)$. Then,
		\begin{equation}
			\label{e:suminfinfek2n0}
			\sum_{n=-\infty}^{\infty}\frac{e^{i2ns}}{(2n-k-2)(2n-k)(2n-k+2)} = -\frac{i\pi}4 e^{iks}\sin^2(s),
		\end{equation}
		and
		\begin{align}
			\label{e:sum0infek2n0}
			& \sum_{n=0}^{\infty}\frac{e^{i2ns}}{(2n-k-2)(2n-k)(2n-k+2)} = \frac{\chi_{(-\infty,-1]}(k)}{k(4-k^2)} 
			\cr
			& \qquad  - e^{iks}\biggg[\sum_{n=0}^{(|k|-1)/2}\frac{e^{-i\sgn(k)(2n+1)s}}{(2n-1)(2n+1)(2n+3)}
			\cr
			& \qquad \qquad \qquad + \frac14\cos(s) + \frac{i\pi}8\sin^2(s) + \frac14\sin^2(s)\ln\left(\cot\left(\frac s2\right)\right)\biggg],
		\end{align}
		where the characteristic function $\chi_{(-\infty,-1]}(k)$ equals $1$ when $k\le -1$, and $0$ when $k > -1$.
		
	\end{lemma}
	
	\begin{proof} The proof of \eqref{e:suminfinfek2n0} is straightforward:
		\begin{align}
			\label{e:suminfinfek2n}
			& \sum_{n=-\infty}^{\infty}\frac{e^{i2ns}}{(2n-k-2)(2n-k)(2n-k+2)} = \sum_{n=-\infty}^{\infty}\frac{e^{i(2n+k+1)s}}{(2n-1)(2n+1)(2n+3)}
			\cr
			& \qquad = e^{iks}\sum_{n=-\infty}^{\infty}\left[\frac{e^{i(2n+1)s}}{8(2n - 1)} - \frac{e^{i(2n+1)s}}{4(2n + 1)} + \frac{e^{i(2n+1)s}}{8(2n + 3)}\right]
			\cr
			& \qquad = e^{iks}\frac{e^{i2s} - 2 + e^{-i2s}}8\sum_{n=-\infty}^{\infty}\frac{e^{i(2n+1)s}}{2n + 1} 
			\cr
			& \qquad = e^{iks}\frac{(e^{is} - e^{-is})^2}8\left[\sum_{n=0}^{\infty}\frac{e^{i(2n+1)s}}{2n + 1} + \sum_{n=-\infty}^{-1}\frac{e^{i(2n+1)s}}{2n + 1}\right]
			\cr
			& \qquad = -e^{iks}\frac{\sin^2(s)}2\left[\sum_{n=0}^{\infty}\frac{e^{i(2n+1)s}}{2n + 1} - \sum_{n=0}^{\infty}\frac{e^{-i(2n+1)s}}{2n + 1}\right] = -\frac{i\pi}4\sin^2(s)e^{iks},
		\end{align}
		where we have applied partial fraction decomposition, and have used \eqref{e:atanh} in the last line. In fact, the last line can be also simplified by noting that it is precisely the Fourier series of a piecewise constant function:
		\begin{align*}
			\sum_{n=0}^{\infty}\frac{e^{i(2n+1)s}}{2n + 1} - \sum_{n=0}^{\infty}\frac{e^{-i(2n+1)s}}{2n + 1} & = 2\sum_{n=0}^{\infty}\frac{\sin((2n+1)s)}{2n + 1}
			\cr
			& =
			\left\{
			\begin{aligned}
				& \frac{\pi}{2}, & & s\in(0,\pi) + 2l\pi,\ \forall l\in\mathbb Z,
				\\
				& {-\frac{\pi}{2}}, & & s\in(-\pi,0) + 2l\pi,\ \forall l\in\mathbb Z.
			\end{aligned}
			\right.
		\end{align*}
		On the other hand, in order to prove \eqref{e:sum0infek2n0} we observe that
		\begin{align}
			\label{e:Sksum0}
			\sum_{n=0}^{\infty}\frac{e^{i2ns}}{(2n-k-2)(2n-k)(2n-k+2)} & = \sum_{n=(-k-1)/2}^{\infty}\frac{e^{i(2n+k+1)s}}{(2n-1)(2n+1)(2n+3)}
			\cr
			& = S_k + \sum_{n=0}^{\infty}\frac{e^{i(2n+k+1)s}}{(2n-1)(2n+1)(2n+3)},
		\end{align}
		where $S_k$ is given by 
		\begin{equation*}
			S_k = 
			\left\{
			\begin{aligned}
				& \sum_{n=(-k-1)/2}^{-1}\frac{e^{i(2n+k+1)s}}{(2n-1)(2n+1)(2n+3)}, & & k \ge 1,
				\cr
				& 0, & & k = -1;
				\cr
				& {-\sum_{n=0}^{(-k-3)/2}}\frac{e^{i(2n+k+1)s}}{(2n-1)(2n+1)(2n+3)}, & & k \le -3.
			\end{aligned}
			\right.
		\end{equation*}
		However, it is possible to express $S_k$ in a more compact way. Indeed, when $k \ge 1$, it can be rewritten as
		$$
		S_k = -e^{iks}\sum_{n=0}^{(k-1)/2}\frac{e^{-i(2n+1)s}}{(2n-1)(2n+1)(2n+3)},
		$$
		and, when $k = -1$ or $k \le -3$, both cases can be considered together:
		$$
		S_k = \frac{1}{k(4-k^2)} - e^{iks}\sum_{n=0}^{(-k-1)/2}\frac{e^{i(2n+1)s}}{(2n-1)(2n+1)(2n+3)}.
		$$
		Therefore, for any odd integer $k$, we get that
		\begin{equation}
			\label{e:Sk}
			S_k = \frac{\chi_{(-\infty,-1]}(k)}{k(4-k^2)} - e^{iks}\sum_{n=0}^{(|k|-1)/2}\frac{e^{-i\sgn(k)(2n+1)s}}{(2n-1)(2n+1)(2n+3)}.
		\end{equation}
		Finally, the infinite sum in \eqref{e:Sksum0} can be computed explicitly. Reasoning as in \eqref{e:suminfinfek2n}:
		\begin{align*}
			& \sum_{n=0}^{\infty}\frac{e^{i(2n+k+1)s}}{(2n-1)(2n+1)(2n+3)} = e^{iks}\sum_{n=0}^{\infty}\left[\frac{e^{i(2n+1)s}}{8(2n - 1)} - \frac{e^{i(2n+1)s}}{4(2n + 1)} + \frac{e^{i(2n+1)s}}{8(2n + 3)}\right]
			\cr
			& \qquad = e^{iks}\left[e^{i2s}\sum_{n=-1}^{\infty}\frac{e^{i(2n+1)s}}{8(2n + 1)} - \sum_{n=0}^{\infty}\frac{e^{i(2n+1)s}}{4(2n + 1)} + e^{-i2s}\sum_{n=1}^{\infty}\frac{e^{i(2n+1)s}}{8(2n + 1)}\right]
			\cr
			& \qquad = e^{iks}\left[-\frac{e^{is} + e^{-is}}{8} + \frac{(e^{is} - e^{-is})^2}{8}\sum_{n=0}^{\infty}\frac{e^{i(2n+1)s}}{2n + 1}\right]
			\cr
			& \qquad = -e^{iks}\left[\frac14\cos(s) + \frac{i\pi}8\sin^2(s) + \frac14\sin^2(s)\ln\left(\cot\left(\frac s2\right)\right)\right],
		\end{align*}
		where we have used again \eqref{e:atanh}. Introducing this expression together with \eqref{e:Sk} into \eqref{e:Sksum0}, we get \eqref{e:sum0infek2n0}, which concludes the proof.
	\end{proof}
	
	\begin{corollary}
		Let $k$ be an odd integer, and $s\in(0,\pi)$. Then,
		\begin{align}
			\label{e:2F1finite}
			& {}_2F_1(1,-k/2-1;-k/2+2;e^{i2s}) = \chi_{(-\infty,-1]}(k) 
			\cr
			& \qquad \qquad  - k(4-k^2)e^{iks}\biggg[\sum_{n=0}^{(|k|-1)/2}\frac{e^{-i\sgn(k)(2n+1)s}}{(2n-1)(2n+1)(2n+3)}
			\cr
			& \qquad \qquad \qquad + \frac14\cos(s) + \frac{i\pi}8\sin^2(s) + \frac14\sin^2(s)\ln\left(\cot\left(\frac s2\right)\right)\biggg].
		\end{align}
	\end{corollary}
	
	\begin{proof} From the definition in \eqref{e:2F1}, and bearing in mind that $(1)_n = n!$,
		\begin{align}
			\label{e:2F1infinite}
			& {}_2F_1(1,-k/2-1;-k/2+2;e^{i2s}) = \sum_{n=0}^{\infty}\frac{(1)_n (-k/2-1)_n e^{i2ns}}{(-k/2+2)_nn!}
			\cr
			& \qquad = \sum_{n=0}^{\infty} \frac{\Gamma(-k/2+n-1)\Gamma(-k/2+2)e^{i2ns}}{\Gamma(-k/2-1)\Gamma(-k/2+n+2)}
			\cr
			& \qquad = k(4-k^2)\sum_{n=0}^{\infty}\frac{e^{i2ns}}{(2n-k-2)(2n-k)(2n-k+2)}.
		\end{align}
		Therefore, multiplying the right-hand side of \eqref{e:sum0infek2n0} by $k(4-k^2)$, we get \eqref{e:2F1finite}.
	\end{proof}
	
	Note that ${}_2F_1$ in \eqref{e:2F1infinite} is \eqref{e:2F1} evaluated at $a = 1$, $b = -k/2-1$, $c = -k/2 + 2$ and $z = e^{i2s}$. Therefore, since $s\in(0,\pi)$, we have $|z| = 1$, $z\not=1$, and $\Re(c - a - b) = 2 > 0$, so the corresponding series in \eqref{e:2F1} is absolutely convergent (see \cite[p. 5]{Slater1966}).
	
	Even if it is long known that certain Gaussian hypergeometric functions whose parameters are integers or halves of integers can be reformulated as finite sums (see, e.g., \cite{detrichconn1979}), to the best of our knowledge, this idea has not been used to develop a very efficient and accurate pseudospectral method for the numerical approximation of the half Laplacian. This will be possible thanks to the following theorem, and to the use of the so-called fast convolution, as will be explained in Section \ref{s:convolution}.
	\begin{theorem}\label{th:fraclap}
		Let $k$ be an odd integer, $L>0$, $s\in(0,\pi)$, and $x = L\cot(s) \Longleftrightarrow s = \acot(x/L)$, where $\acot(x)$ takes values in $(0,\pi)$. Then, $(-\Delta)^{1/2}e^{ik\acot(x/L)} \equiv (-\Delta)_s^{1/2}e^{iks}$ is given by
		\begin{align}
			\label{e:tha}
			(-\Delta)_s^{1/2}e^{iks} = {-\frac{2i}{L\pi (k+2)}} - \frac{k}{L}\sin^2(s)e^{iks} + \frac{8i\,{}_2F_1(1,-k/2-1;-k/2+2;e^{i2s})}{L\pi (4-k^2)},
		\end{align}
		or, equivalently,
		\begin{align}
			\label{e:thb}
			(-\Delta)_s^{1/2}e^{iks} & = {-\frac{2i\sgn(k)}{L\pi (|k|+2)}} - \frac{2ik}{L\pi}e^{iks}\biggg[\cos(s) + \sin^2(s)\ln\left(\cot\left(\frac s2\right)\right)
			\cr
			& \qquad + \sum_{n=0}^{(|k|-1)/2}\frac{4e^{-i\sgn(k)(2n+1)s}}{(2n-1)(2n+1)(2n+3)}\biggg].
		\end{align}
		Furthermore, when $L = 1$, \eqref{e:tha} and \eqref{e:thb} are expressed in terms of $x$ as
		\begin{align}
			\label{e:thc}
			& (-\Delta)^{1/2} \frac{(x + i)^k}{(1+x^2)^{k/2}} = {-\frac{2i}{\pi (k+2)}} - \frac{k(x + i)^k}{(1+x^2)^{k/2+1}}
			\cr
			& \qquad \qquad + \frac{8i}{\pi (4-k^2)}{}_2F_1\left(1,-k/2-1;-k/2+2; \frac{(x+i)^2}{1+x^2} \right)
			\cr
			& \qquad = {-\frac{2i\sgn(k)}{\pi (|k|+2)}} - \frac{2ik(x + i)^k[x(1+x^2)^{1/2} + \arg\sinh(x)]}{\pi(1+x^2)^{k/2+1}}
			\cr
			& \qquad \qquad - \sum_{n=0}^{(|k|-1)/2}\frac{8ik(x + i)^{[k-\sgn(k)(2n+1)]}}{\pi(2n-1)(2n+1)(2n+3)(1+x^2)^{[k-\sgn(k)(2n+1)]/2}}.
		\end{align}
		
	\end{theorem}
	
	\begin{proof} In order to simplify the notation, we assume, without loss of generality, that $L = 1$; otherwise, it is enough to divide $(-\Delta)_s^{1/2}e^{iks}$ by $L$. When $k$ is odd, from \eqref{e:tha1},
		\begin{align*}
			(-\Delta)_s^{1/2}e^{iks} & = \frac{ik}{\pi}\Bigg[\frac{2}{4 - k^2} + \sum_{n=0}^\infty\frac{4e^{i2ns}}{(2n - k - 2)(2n - k)(2n - k + 2)}
			\cr
			& \qquad + \sum_{n=0}^{\infty}\frac{4e^{-i2ns}}{(2n+k-2)(2n+k)(2n+k+2)}\Bigg]
			\cr
			& = \frac{ik}{\pi}\Bigg[\frac{2}{4 - k^2} + \sum_{n=0}^\infty\frac{(4+4)e^{i2ns}}{(2n - k - 2)(2n - k)(2n - k + 2)}
			\cr
			& \qquad - \frac{4}{k(4 - k^2)} - \sum_{n=-\infty}^\infty\frac{4e^{i2ns}}{(2n - k - 2)(2n - k)(2n - k + 2)}\Bigg],
		\end{align*}
		so \eqref{e:tha} follows after applying \eqref{e:suminfinfek2n0} and \eqref{e:sum0infek2n0}. On the other hand, introducing \eqref{e:2F1finite} into \eqref{e:tha},
		\begin{align*}
			(-\Delta)_s^{1/2}e^{iks} & = {-\frac{2i}{\pi (k+2)}} +
			\frac{8i\,\chi_{(-\infty,-1]}(k)}{\pi (4-k^2)} - \frac{2ik}{\pi}e^{iks}\biggg[\cos(s) 
			\cr
			& \qquad + \sin^2(s)\ln\left(\cot\left(\frac s2\right)\right) + \sum_{n=0}^{(|k|-1)/2}\frac{4e^{-i\sgn(k)(2n+1)s}}{(2n-1)(2n+1)(2n+3)}\biggg].
		\end{align*}
		Finally, when $k \ge 1$,
		$$
		{-\frac{2i}{\pi (k+2)}} + \frac{8i\,\chi_{(-\infty,-1]}(k)}{\pi (4-k^2)} = {-\frac{2i}{\pi (k+2)}} = {-\frac{2i\sgn(k)}{\pi (|k|+2)}},
		$$
		and, when $k\le -1$,
		$$
		{-\frac{2i}{\pi (k+2)}} + \frac{8i\,\chi_{(-\infty,-1]}(k)}{\pi (4-k^2)} = \frac{2i}{\pi (-k+2)} = {-\frac{2i\sgn(k)}{\pi (|k|+2)}},
		$$
		which concludes the proof of \eqref{e:thb}. The proof of \eqref{e:thc} follows immediately from \eqref{e:tha} and \eqref{e:thb}, after applying the following identities:
		\begin{equation*}
			e^{iks} \equiv \frac{(x + i)^k}{(1+x^2)^{k/2}}, \qquad \cos(s) \equiv \frac{x}{(1+x^2)^{1/2}}, \qquad \sin^2(s) \equiv \frac{1}{1+x^2}
		\end{equation*}
		and
		$$
		\ln\left(\cot\left(\dfrac{s}{2}\right)\right) \equiv \ln(x + (x ^ 2+1)^{1/2}) \equiv \arg\sinh(x).
		$$
	\end{proof}
	
	Note that there are other equivalent ways of expressing \eqref{e:tha}, \eqref{e:thb} and \eqref{e:thc}. For instance, using the fact that $(-\Delta)_s^{1/2}e^{iks} = \overline{(-\Delta)_s^{1/2}e^{-iks}}$, \eqref{e:tha} becomes
	\begin{equation*}
		(-\Delta)_s^{1/2}e^{iks} = {-\frac{2i}{L\pi (k-2)}} + \frac{k}{L}\sin^2(s)e^{iks} - \frac{8i\,{}_2F_1(1,k/2-1;k/2+2;e^{-i2s})}{L\pi (4-k^2)}.
	\end{equation*}
	
	\begin{corollary}Let $k$ be an odd integer, $L>0$, $s\in(0,\pi)$, and $x = L\cot(s) \Longleftrightarrow s = \acot(x/L)$, where $\acot(x)$ takes values in $(0,\pi)$. Then,
		\begin{align}
			\label{e:deltas12cosks}
			(-\Delta)_s^{1/2}\cos(ks) & = \frac{2k}{L\pi}\sin(ks)\left(\cos(s) + \sin^2(s)\ln\left(\cot\left(\frac s2\right)\right)\right) 
			\cr
			& \qquad + \frac{8k}{L\pi}\sum_{n=0}^{(k-1)/2}\frac{\sin((k-1-2n)s)}{(2n-1)(2n+1)(2n+3)},
		\end{align}
		and
		\begin{align}
			\label{e:deltas12sinks}
			(-\Delta)_s^{1/2}\sin(ks) & = {-\frac{2}{L\pi (k+2)}} - \frac{2k}{L\pi}\cos(ks)\left(\cos(s) + \sin^2(s)\ln\left(\cot\left(\frac s2\right)\right)\right) 
			\cr
			& \qquad - \frac{8k}{L\pi}\sum_{n=0}^{(k-1)/2}\frac{\cos((k-1-2n)s)}{(2n-1)(2n+1)(2n+3)}.
		\end{align}
		
		\begin{proof} \eqref{e:deltas12cosks} and \eqref{e:deltas12sinks} are respectively the real part and the imaginary part of \eqref{e:thb}. Note that these expressions can be immediately formulated in terms of $x$, by choosing $L = 1$ and taking respectively the real part and the imaginary part of \eqref{e:thc}.
		\end{proof}
		
	\end{corollary}
	
	\section{A fast convolution result}
	
	\label{s:convolution}
	
	Given a sequence $b : \mathbb Z\to\mathbb C$ of complex numbers, which we denote as $\{b_l\}$, and a number $P\in\mathbb N$, we say that $\{b_l\}$ is $P$-periodic if $b_{l + P} = b_l$, for all $l\in\mathbb Z$. Therefore, in order to define a $P$-periodic sequence, just $P$ adjacent values are needed, e.g., $\{b_0, b_1, \ldots, b_{P-1}\}$.
	
	Recall that, given a $P$-periodic sequence $\{b_l\}$ of complex numbers, its discrete Fourier transform $\{\hat b_p\}$ is defined from the $P$ values $\{b_0, b_1, \ldots, b_{P-1}\}$ of the first period of $\{b_l\}$ as
	\begin{equation}
		\label{e:FFT}
		\hat{b}_{p} = \sum_{l=0}^{P-1}b_{l}e^{-i\pi2lp/P},\quad p\in\mathbb Z.
	\end{equation}
	Then, it is straightforward to check that $\{\hat b_p\}$ is also a $P$-periodic sequence. Moreover, the original numbers $\{b_0, b_1, \ldots, b_{P-1}\}$ and their periodic extension can be recovered by means of the inverse discrete Fourier transform of the $P$ values $\{\hat b_0, \hat b_1, \ldots, \hat b_{P-1}\}$ of the first period of $\{\hat b_p\}$:
	\begin{equation}
		\label{e:IFFT}
		b_l = \frac{1}{P}\sum_{p=0}^{P-1}\hat{b}_{p}e^{i\pi2lp/P},\quad p\in\mathbb Z.
	\end{equation}
	Note that, in principle, a direct implementation of \eqref{e:FFT} and \eqref{e:IFFT} would give us a computational cost of the order of $\mathcal O(P^2)$. However, it is very important to underline that both \eqref{e:FFT} and \eqref{e:IFFT} can be computed in just $\mathcal O(P\ln(P))$ operations by means of the fast Fourier transform (FFT) and the inverse fast Fourier transform (IFFT), respectively (see \cite{FFT}).
	
	On the other hand, if we consider the convolution of two $P$-periodic sequences $\{b_l\}$ and $\{c_l\}$:
	\begin{equation}
		\label{e:defconf}
		(b\ast c)_{l} = \sum_{n=0}^{P-1}b_{l}c_{n-l},\quad l\in\mathbb Z;
	\end{equation}
	then $(b\ast c)_{l} = (c\ast b)_{l}$ is a $P$-periodic sequence satisfying
	\begin{equation}
		\label{e:convprop}
		(\widehat{b\ast c})_{p} = \hat{b}_{p}\hat{c}_{p}, \quad p\in\mathbb Z.
	\end{equation}
	This very important property, whose proof is straightforward (see, e.g., \cite{cayamacuestadelahozgarciacervera2022,GarciaCervera2007}), is known as the discrete convolution theorem. Thanks to it, \eqref{e:defconf} can be computed in only $\mathcal O(P\ln(P))$ operations, because just two FFTs and one IFFT are required, whereas a direct calculation of \eqref{e:defconf} without \eqref{e:convprop} would instead need $\mathcal O(P^2)$ operations, i.e., it would be much more expensive from a computational point of view.
	
	However, if $\{b_l\}$ and $\{c_l\}$ are not periodic, we cannot apply directly \eqref{e:convprop}. For instance, in Lemma \ref{lemma:fast}, given $M \in \mathbb N$, we will need to compute a convolution of this form:
	\begin{equation}
		\label{e:bastc}
		(b\ast c)_l = \sum_{n = 0}^Mb_nc_{l-n}, \quad 0 \le l \le M,
	\end{equation}
	where $\{b_l\}$ and $\{c_l\}$ are not $M+1$-periodic sequences, but finite ones; indeed, only the values $\{b_0, \ldots, b_M\}$ and $\{c_{-M}, \ldots, c_M\}$ are known.
	Note that, even if $\{b_l\}$ poses no problems and can be regarded as $M+1$ periodic, we cannot do such thing with $\{c_l\}$, because $c_{-M}\not=c_{1}, \ldots, c_{-1}\not=c_{M}$, so applying the FFT to only $\{c_0, \ldots, c_M\}$ would miss the data $\{c_{-M}, \ldots, c_{-1}\}$ and hence would not be enough to compute correctly \eqref{e:bastc}. Therefore, following the ideas in \cite{cayamacuestadelahozgarciacervera2022,GarciaCervera2007}, we take an integer $P \ge 2M+1$, and extend them as follows:
	\begin{equation}
		\label{e:tildebl0}
		\tilde b_l =
		\begin{cases}
			b_l, & 0 \le l \le M,
			\\
			0, & M + 1\le l\le P-1,
		\end{cases}
	\end{equation}
	and
	\begin{equation}
		\label{e:tildecl0}
		\tilde c_l = \begin{cases}
			c_{l}, & 0 \le l\le M, \\
			0, & M + 1\le l \le P - M - 1,
			\\
			c_{l-P}, & P - M \le l \le P - 1.
		\end{cases}
	\end{equation}
	Then, we have immediately that
	\begin{equation}
		\label{e:bastcP}
		(b\ast c)_l \equiv \sum_{n=0}^{M-1}b_{n}c_{l-n} = \sum_{n=0}^{P-1}\tilde b_{n}\tilde c_{l-n} = (\tilde b\ast \tilde c)_{l}, \quad l = 0, \ldots, M,
	\end{equation}
	Evidently, $\{\tilde b_0, \ldots, \tilde b_{P-1}\}$ and $\{\tilde c_0, \ldots, \tilde c_{P-1}\}$ are still finite sequences, but we can regard them now as $P$-periodic. Therefore, in \eqref{e:bastcP}, when $\{\tilde c_{-M}, \ldots, \tilde c_{-1}\}$ intervene, assuming that $\{\tilde c_l\}$ is $P$-periodic implies that $\tilde c_{-M} \equiv \tilde c_{P-M}, \ldots, \tilde c_{-1} \equiv \tilde c_{P-1}$, but, from the definition of $\{\tilde c_l\}$ in \eqref{e:tildecl0}, $\tilde c_{P-M} \equiv c_{-M}, \ldots, \tilde c_{P-1} \equiv c_{-1}$, so $\tilde c_{-M} \equiv c_{-M}, \ldots, \tilde c_{-1} \equiv c_{-1}$, as is required. Hence, $(\tilde b\ast \tilde c)_{l}$, for $0 \le l \le P-1$, can be computed by means of a fast convolution, and then we keep $(b\ast c)_l \equiv (\tilde b\ast \tilde c)_{l}$, for $0\le l\le M$, and ignore $(\tilde b\ast \tilde c)_{l}$, for $M+1\le l\le P-1$. This idea is used in the following lemma to compute efficiently linear combinations of $(-\Delta)_s^{1/2}e^{iks}$, for $k \in\{1, 3, \ldots, 2M+1\}$, and for $k \in\{-1, -3, \ldots, -2M-1\}$.
	
	\begin{lemma} \label{lemma:fast} Let $M\in\mathbb N$, $a_0, a_1, \ldots, a_M\in\mathbb C$, and $s\in(0,\pi)$. Then,
		\begin{align}
			\label{e:sumeiks}
			& \sum_{l = 0}^Ma_l(-\Delta)_s^{1/2}e^{i(2l+1)s} = {-\frac{2i}{L\pi}}\sum_{l = 0}^M\frac{a_l}{2l+3}  + \frac{2i}{L\pi}\sum_{l = 0}^M(\tilde b\ast\tilde c)_le^{i2ls}
			\cr
			& \qquad - \frac{2i}{L\pi}\Big[\cos(s) + \sin^2(s)\ln\left(\cot\left(\frac s2\right)\right)\Big]\sum_{l = 0}^M(2l+1)a_le^{i(2l+1)s},
		\end{align}
		where $P$ can be any natural number such that $P \ge 2M+1$, and $(\tilde b\ast\tilde c)_l$ denotes the discrete convolution of the $P$-periodic sequences $\tilde b_l$ and $\tilde c_l$: 
		\begin{equation}
			\label{e:tildebconvtildec}
			(\tilde b\ast\tilde c)_l = \sum_{n = 0}^P\tilde b_n\tilde c_{l-n},
		\end{equation}
		with $\tilde b_l$ and $\tilde c_l$ being given respectively by
		\begin{equation}
			\label{e:tildebl}
			\tilde b_l =
			\left\{\begin{aligned}
				& (8l+4)a_l, & & 0 \le l \le M,
				\\
				& 0, & & M + 1 \le l \le P - 1,
			\end{aligned}
			\right.
		\end{equation}
		and
		\begin{equation}
			\label{e:tildecl}
			\tilde c_l = \left\{\begin{aligned}
				& \frac13, & & l = 0, \\
				& 0, & & 1 \le l \le P-M-1, \\
				& \frac{1}{(2(l-P)-3)(2(l-P)-1)(2(l-P)+1)}, & & P - M \le l \le P-1.
			\end{aligned}\right.
		\end{equation}
		
	\end{lemma}
	
	\begin{proof} We consider a linear combination of  $(-\Delta)_s^{1/2}e^{iks}$, for $k = 1, \ldots, 2M+1$. From \eqref{e:thb} in Theorem \ref{th:fraclap},
		\begin{align*}
			& \sum_{l = 0}^Ma_l(-\Delta)_s^{1/2}e^{i(2l+1)s} = \sum_{l = 0}^Ma_l\biggg({-\frac{2i}{L\pi (2l+3)}} - \frac{2i(2l+1)}{L\pi}e^{i(2l+1)s}\biggg[\cos(s)
			\cr
			& \qquad \qquad + \sin^2(s)\ln\left(\cot\left(\frac s2\right)\right)  + \sum_{n=0}^{l}\frac{4e^{-i(2n+1)s}}{(2n-1)(2n+1)(2n+3)}\biggg]\biggg)
			\cr
			& \qquad = \Bigg[{-\frac{2i}{L\pi}}\sum_{l = 0}^M\frac{a_l}{2l+3} - \frac{2i}{L\pi}\left(\cos(s) + \sin^2(s)\ln\left(\cot\left(\frac s2\right)\right)\right)
			\cr
			& \qquad \qquad \times\sum_{l = 0}^M(2l+1)a_le^{i(2l+1)s}\Bigg] + \left[{-\frac{2i}{L\pi}}\sum_{l = 0}^M\sum_{n=0}^{l}\frac{a_l(8l+4)e^{i2(l-n)s}}{(2n-1)(2n+1)(2n+3)}\right]
			\cr
			& \qquad = [\textrm{I}]  + [\textrm{II}],
		\end{align*}
		where $[\textrm{I}]$ corresponds to the first and last terms in \eqref{e:sumeiks}, and $[\textrm{II}]$, which corresponds to the second term in \eqref{e:sumeiks}, can be expressed as a convolution. More precisely, replacing $n$ by $l - n$, changing the order of summation, introducing the characteristic function $\chi_{(-\infty,0]}(x)$, which equals $1$ when $x\le 0$, and $0$ when $x > 0$, and swapping $n$ and $l$, it becomes
		\begin{align*}
			[\textrm{II}] & = {-\frac{2i}{L\pi}}\sum_{l = 0}^M\sum_{n=0}^{l}\frac{a_l(8l+4)e^{i2ns}}{(2(l-n)-1)(2(l-n)+1)(2(l-n)+3)}
			\cr
			& = \frac{2i}{L\pi}\sum_{n = 0}^M\left[\sum_{l=n}^{M}\frac{-a_l(8l+4)}{(2(l-n)-1)(2(l-n)+1)(2(l-n)+3)}\right]e^{i2ns}
			\cr
			& = \frac{2i}{L\pi}\sum_{n = 0}^M\left[\sum_{l=0}^{M}\frac{a_l(8l+4)\chi_{(-\infty,0]}(n-l)}{(2(n-l)-3)(2(n-l)-1)(2(n-l)+1)}\right]e^{i2ns}
			\cr
			& = \frac{2i}{L\pi}\sum_{n = 0}^M\left[\sum_{l=0}^{M}b_lc_{n-l}\right]e^{i2ns} = \frac{i}{L\pi}\sum_{l = 0}^M\left[\sum_{n=0}^{M}b_nc_{l-n}\right]e^{i2ls} = \frac{i}{L\pi}\sum_{l = 0}^M(b\ast c)_le^{i2ls},
		\end{align*}
		where $(b\ast c)_l$, for $l\in\{0, \ldots, M\}$, denotes the discrete convolution between the sequences $b_l$ and $c_l$, which are given respectively by
		\begin{align*}
			b_l & = (8l+4)a_l,
			\cr
			c_l & = \frac{\chi_{(-\infty, 0]}(l)}{(2l-3)(2l-1)(2l+1)} =
			\left\{
			\begin{aligned}
				& 0, & & l > 0,
				\cr
				& \frac{1}{(2l-3)(2l-1)(2l+1)}, & & l \le 0.
			\end{aligned}
			\right.
		\end{align*}
		At this point, since neither $b_l$ nor $c_l$ are periodic, we extend them to form $P$-periodic sequences $\tilde b_l$ and $\tilde c_l$, with $P \ge 2M+1$, by means of \eqref{e:tildebl0} and \eqref{e:tildecl0}, obtaining precisely \eqref{e:tildebl} and \eqref{e:tildecl}. Then,  $(\tilde b\ast \tilde c)_l = (b\ast c)_l$, and it follows that
		$$
		[\textrm{II}] = \frac{2i}{L\pi}\sum_{l = 0}^M(\tilde b\ast \tilde c)_le^{i2ls},
		$$
		which concludes the proof of \eqref{e:sumeiks}.
	\end{proof}

	\begin{corollary}
		Let $M\in\mathbb N$, $a_0, a_1, \ldots, a_M\in\mathbb C$, and $s\in(0,\pi)$. Then,
		\begin{align}
			\label{e:sumeiksneg}
			& \sum_{l = 0}^Ma_l(-\Delta)_s^{1/2}e^{-i(2l+1)s} = \frac{2i}{L\pi}\sum_{l = 0}^M\frac{a_l}{2l+3}  - \frac{2i}{L\pi}\sum_{l = 0}^M(\tilde b\ast\tilde c)_le^{-i2ls}
			\cr
			& \qquad + \frac{2i}{L\pi}\Big[\cos(s) + \sin^2(s)\ln\left(\cot\left(\frac s2\right)\right)\Big]\sum_{l = 0}^M(2l+1)a_le^{-i(2l+1)s},
		\end{align}
		where $(\tilde b\ast\tilde c)_l$ is defined in \eqref{e:tildebconvtildec}, and $\tilde b_l$ and $\tilde c_l$ are given respectively by \eqref{e:tildebl} and \eqref{e:tildecl}.
		
	\end{corollary}
	
	\begin{proof}  Bearing in mind that $(-\Delta)_s^{1/2}e^{-i(2l+1)s} = \overline{(-\Delta)_s^{1/2}e^{i(2l+1)s}}$, the proof of \eqref{e:sumeiksneg} is identical to that of \eqref{e:sumeiks}.
		
	\end{proof}
	
	The fact that $(\tilde b\ast \tilde c)_l$ in \eqref{e:sumeiks} and \eqref{e:sumeiksneg}, as defined in \eqref{e:tildebconvtildec}, can be computed by means of a fast convolution in $\mathcal{O}(P\ln(P))$ operations will allow us to develop an efficient and accurate pseudosprectral method for the numerical computation of the half Laplacian.
	
	\section{Numerical implementation}\label{s:implementation}
	
	We consider two different situations: when $U(s)$ is periodic of period $\pi$, and when it is not, where, by $U(s)$ being periodic, we understand that $U(0) = U(\pi)$, or, equivalently, in terms of $x$, that $\lim_{x\to-\infty}u(x) = \lim_{x\to\infty}u(x)$.
	
	\subsection{Case with $U(s)$ periodic of period $\pi$}\label{s:periodic} Although this case is a straightforward application of \eqref{e:tha1} for $k$ even and poses no problems, we cover it for the sake of completeness, and also because it is indeed possible to reduce the non-periodic case to the periodic one, as we will show in Section \ref{s:nonperiodictoperiodic}.
	
	Indeed, since the period is $\pi$, then, $U(s)$ can be expanded as
	$$
	U(s) = \sum_{k = -\infty}^{\infty}\hat U(k)e^{i2ks},
	$$
	However, we cannot deal with the infinitely many frequencies, so we approximate it as
	\begin{equation}
		\label{e:U(s)}
		U(s) \approx \sum_{k = -N/2}^{N/2-1}\hat U(k)e^{i2ks},
	\end{equation}
	and, following a pseudospectral approach \cite{trefethen}, impose the equality at $N$ different nodes. In our case, we take $N$ equally-spaced nonterminal nodes:
	\begin{equation}
		\label{e:sj}
		s_{j}=\frac{\pi(2j+1)}{2N}\Longleftrightarrow x_{j} = L\cot\left(\frac{\pi(2j+1)}{2N}\right), \quad j\in\{0, \ldots, N-1\},
	\end{equation}
	so $s_0 = \pi / (2N)$, $s_{N-1} = \pi - \pi / (2N)$, $s_{j+1} - s_j = \pi/N$, for all $j$. Evaluating \eqref{e:U(s)} at \eqref{e:sj} and imposing the equality:
	\begin{align}
		\label{e:usjfourier}
		U(s_j) & = \sum_{k = -\lfloor N/2\rfloor}^{\lceil N/2\rceil-1}\hat U(k)e^{i2ks_j} = \sum_{k = -\lfloor N/2\rfloor}^{\lceil N/2\rceil-1}\hat U(k)e^{i\pi k(2j+1)/N}
		\cr
		& = \sum_{k = 0}^{\lceil N/2\rceil-1}\left[e^{i\pi k/N}\hat U(k)\right]e^{i\pi 2 jk/N} + \sum_{k = \lceil N/2\rceil}^{N-1}\left[e^{i\pi (k-N)/N}\hat U(k - N)\right]e^{i\pi 2 jk/N},
	\end{align}
	where $\lfloor\cdot\rfloor$ and $\lceil\cdot\rceil$ denote respectively the floor and ceil functions, so, according to \eqref{e:IFFT}, the values $\{U(s_j)\}$ are precisely the inverse discrete Fourier transform of
	\begin{align*}
		\Big\{N\hat U(0), Ne^{i\pi/N}\hat{U}(1), \ldots, & Ne^{i(\lceil N/2\rceil-1)\pi/N}\hat{U}(\lceil N/2\rceil-1),
		\cr
		& Ne^{-i\lfloor N/2\rfloor\pi/N}\hat{U}(-\lfloor N/2\rfloor), \ldots, Ne^{-i\pi/N}\hat{U}(-1)\Big\};
	\end{align*}
	and, conversely,
	\begin{equation}
		\label{e:hatukfourier}
		\hat{U}(k) \equiv \frac{e^{-i\pi k/N}}{N}\sum_{j=0}^{N-1}U(s_{j})e^{-2ijk\pi/N},
	\end{equation}
	i.e., according to \eqref{e:FFT}, the Fourier coefficients $\hat U(k)$ are given by computing the discrete Fourier transform of $\{U(s_0), \ldots, U(s_{N-1})\}$, and multiplying the result by $e^{-ik\pi/N}/N$, for $k\in\{-\lfloor N/2\rfloor, \ldots, \lceil N/2\rceil-1\}$. Therefore, both \eqref{e:usjfourier} and \eqref{e:hatukfourier} can be obtained very efficiently by means of the IFFT and FFT, respectively. On the other hand, we apply systematically a Krasny filter \cite{krasny}, i.e., we set to zero all the Fourier coefficients $\hat U(k)$ with modulus smaller than a fixed threshold, which in this paper is the epsilon of the machine, namely $\varepsilon = 2^{-52}$.
	
	On the other hand, applying $(-\Delta)_s^{1/2}$ to \eqref{e:U(s)}, it follows from the first case in \eqref{e:tha1} (where we have substituted $k$ by $2k$) that
	\begin{equation*}
		(-\Delta)_s^{1/2}U(s) \approx \sum_{k = -\lfloor N/2\rfloor}^{\lceil N/2\rceil-1}\hat U(k)(-\Delta)_s^{1/2}e^{i2ks} = \frac2L\sin^2(s)\sum_{k = -\lfloor N/2\rfloor}^{\lceil N/2\rceil-1}|k|\hat U(k)e^{i2ks}.
	\end{equation*}
	Evaluating it at $s = s_j = \pi(2j+1)/(2N)$:
	\begin{equation}
		\label{e:fraclapeven}
		(-\Delta)_s^{1/2}U(s_j) \approx \frac2L\sin^2(s_j)\sum_{k = -\lfloor N/2\rfloor}^{\lceil N/2\rceil-1}\left[e^{i\pi k/N}|k|\hat U(k)\right]e^{i\pi2jk/N}.
	\end{equation}
	From an implementational point of view in Matlab, note that the commands \verb"fft" and \verb"ifft" match exactly \eqref{e:FFT} and \eqref{e:IFFT}. Therefore, if the variable \verb"u" stores $\{U(s_0), \ldots, U(s_{N-1})\}$, we must multiply the FFT of \verb"u" by $e^{-ik\pi/N}/N$ to approximate $\hat U(k)$, then multiply each $\hat U(k)$ by $|k|$, and also by $Ne^{ik\pi/N}$, before computing the IFFT; finally, we multiply the result by the factor outside the sum, namely $2\sin^2(s_j)/L$. Therefore, $e^{-ik\pi/N}/N$ and $Ne^{ik\pi/N}$ cancel each other, and the actual code is reduced to a few lines:
	
\begin{lstlisting}[language=Matlab, basicstyle={\footnotesize\ttfamily}, caption = {Half Laplacian of $u(x)$, when $U(s)$ is periodic of period $\pi$}]
clear
tic
N=10000019; % number of nodes
L=1.1; % scale factor
s=pi*(1:2:2*N-1)/(2*N); % s_j\in(0,\pi)
x=L*cot(s); % x_j\in(-\infty,\infty)
u=1./(1+x.^4); % u(x_j)
Fu=((1-x.^2).*(x.^4+4*x.^2+1)./(1+x.^4).^2)/sqrt(2); % (-\Delta)^{1/2}u(x_j)
u_=fft(u); % compute the FFT of u
u_(abs(u_)/N<eps)=0; % apply Krasny's filter
k=[0:ceil(N/2)-1,-floor(N/2):-1]; % define the frequencies
Fu_num=(2*sin(s).^2/L).*ifft(abs(k).*u_); % numerical approximation of Fu
if isreal(u), Fu_num=real(Fu_num); end % if u is real, so is Fu_num
norm(Fu_num-Fu,"inf") % error in discrete $L^\infty$-norm
toc % ellapsed time
\end{lstlisting}
	
	In this example, we have approximated numerically the half Laplacian of $u(x) = (1+x^4)^{-1}$, whose exact expression is given by
	\begin{align}
		\label{e:fraclap11x4}
		(-\Delta)^{1/2}\frac{1}{1+x^4} = \frac{(1-x^2)(1+4x^2+x^4)}{\sqrt2(1+x^4)^2}.
	\end{align}
	This can be deduced by differentiating the Hilbert transform of $u(x) = (1+x^4)^{-1}$ (see, e.g., \cite{weideman1995}):
	$$
	\mathcal H\left(\frac{1}{1+x^4}\right) = \frac{x(1+x^2)}{\sqrt2(1+x^4)},
	$$
	or, with minimal rewriting, by means of  Mathematica \cite{mathematica}, using \eqref{e:fraclaplmathematica}:
	\begin{verbatim}
		u[x_] = 1/(1 + x^4); Assuming[Im[x] == 0, 
		Integrate[(u'[x - y] - u'[x + y])/y, {y, 0, Infinity}]/Pi]
	\end{verbatim}
	Moreover, it is possible to obtain \eqref{e:fraclap11x4} by using complex variable techniques, as is shown in Appendix \ref{s:appendix}.
	
	We have taken $L=1.1$, and $N = 10000019$, i.e., the first prime number larger than $10^7$. Then, the code requires $4.34$ seconds to execute and the error is $1.6542\times10^{-14}$. Evidently, whether $N$ is a prime number or not is irrelevant from the point of view of the algorithm itself. However, since the lion's share of the computational cost lies on the FFT (which, in the case of Matlab, uses the FFTW library \cite{FFT}), and it is a well known fact that the FFT is especially efficient when the size of the data is a product of powers of small prime numbers, we have deemed interesting to see how our codes perform when taking large prime numbers. In this regard, the results show that, under Matlab, $N$ does not need to have any special property, and, indeed, the program is able to work efficiently with extremely large values of $N$. Nevertheless, to get a more accurate idea of the cost of invoking the FFT with respect to the total cost, we have modified slightly the code, by adding a variable \verb|tFFT| that stores the times. Then, since we are interested in obtaining only the execution time of the \verb|fft| and \verb|ifft| commands, we do not measure the cost of, e.g., the whole line \verb|Fu_num=(2*sin(s).^2/L).*ifft(abs(k).*u_);|, but rather store \verb|abs(k).*u_| in a variable, e.g., \verb|aux_=abs(k).*u_|; and then measure the cost of computing the IFFT of \verb|aux_|, by typing \verb|tFFT=tFFT-toc; aux=ifft(aux_); tFFT=tFFT+toc;|. In this way, we conclude that the calls to the commands \verb|fft| and \verb|ifft| require $3.38$ seconds, i.e., approximately $77\%$ of the total time, and the non-FFT parts, $0.96$ seconds.
	
	On the other hand, it is true that, when $N$ is, e.g., a power of two, the FFT executes faster, which translates into a globally smaller elapsed time. For instance, if we take, e.g., $N = 2^{24} = 16777216$, i.e., a larger number, but which is a power of $2$, then only $2.11$ seconds are required, and the error is $1.5321\times10^{-14}$. However, invoking \verb|fft| and \verb|ifft| requires now only $0.65$ seconds, i.e., approximately $30\%$ of the total time, and the non-FFT parts, $1.46$ seconds. Note also that, with respect to the previous experiment with $N = 10000019$, the time corresponding to the non-FFT parts has grown roughly in a proportional way, i.e., $16777216 / 10000019 = 1.67$, and $1.46 / 0.96 = 1.52$. Therefore, if multiple executions of the program are to be done, or if the codes are to be run in systems with a less tuned FFT implementation, it may be convenient to choose an $N$ which is the product of powers of small prime numbers. Note that, in any case, the accuracy does not degrade when $N$ is extremely large.
	
	Finally, observe that, if $u(x)$ is real, as in this example, we impose explicitly that $(-\Delta)^{1/2}u(x)$ is real, because the numerical approximation might introduce infinitesimal values in the imaginary part.
	
	\subsection{Case with $U(s)$ not (necessarily) periodic of period $\pi$} \label{s:implementationnotperiodic} The implementation of this case is more involved, and constitutes another central point of this paper. If $U(s)$ is not periodic of period $\pi$, we extend it to $s\in[0,2\pi]$; even if there are infinitely many ways of doing it, an even extension at $s = \pi$, i.e., such that $U(\pi + s) \equiv U(\pi - s)$ is usually enough (in Section \ref{s:experiments}, we will consider other kinds of extensions). Then, $U(s)$ can be expanded as a cosine series, and hence $u(x)$ can be expanded in terms of the rational Chebyshev functions $TB_k(x) =  T_k(x / \sqrt{x^2+L^2})$, where $T_k(x) = \cos(k\arccos(x))$ is the $k$th Chebyshev polynomial (see, e.g., \cite{Boyd1987}). On the other hand, in this paper, given the fact that we know how to compute $(-\Delta)_s^{1/2}e^{iks}$, we work with $s$, approximating $U(s)$ as
	\begin{align}
		\label{e:U(s)2}
		U(s) \approx \sum_{k = -N}^{N-1}\hat{U}(k)e^{iks}.
	\end{align}
	Then, we impose the equality at the nodes $s_j$ defined in \eqref{e:sj}, but taking $j\in\{0, \ldots, 2N-1\}$:
	\begin{align*}
		U(s_j) & \equiv \sum_{k=-N}^{N-1}\hat{U}(k)e^{iks_j} = \sum_{k=-N}^{N-1}\hat{U}(k)e^{ik\pi(2j+1)/(2N)}
		\cr
		& = \sum_{k=0}^{N-1}\left[\hat{U}(k)e^{ik\pi/(2N)}\right]e^{2ijk\pi/(2N)}
		\cr
		& \qquad + \sum_{k=N}^{2N-1}\left[\hat{U}(k-2N)e^{i(k-2N)\pi/(2N)}\right]e^{2ijk\pi/(2N)},
	\end{align*}
	so the values $\{U(s_j)\}$ are now the inverse discrete Fourier transform of
	\begin{align*}
		\Big\{2N\hat U(0), 2Ne^{i\pi/(2N)}\hat{U}(1), \ldots, & 2Ne^{i(N-1)\pi/(2N)}\hat{U}(N-1),
		\cr
		& 2Ne^{-iN\pi/(2N)}\hat{U}(-N), \ldots, 2Ne^{-i\pi/(2N)}\hat{U}(-1)\Big\};
	\end{align*}
	and, conversely,
	\begin{equation*}
		\hat{U}(k) \equiv \frac{e^{-ik\pi/(2N)}}{2N}\sum_{j=0}^{2N-1}U(s_{j})e^{-2ijk\pi/(2N)}.
	\end{equation*}
	On the other hand, in order to approximate $(-\Delta)_s^{1/2}U(s)$, we apply $(-\Delta)_s^{1/2}$ to \eqref{e:U(s)2}, to get:
	\begin{align}
		\label{e:deltas12sum}
		(-\Delta)_s^{1/2}U(s) \approx \sum_{k = -N}^{N-1}\hat{U}(k)(-\Delta)_s^{1/2}e^{iks}.
	\end{align}
	Then, we decompose the sum, by distinguishing between odd and even positive and even negative values of $k$,  for which we observe that
	$$
	\{-N, \ldots, N-1\} = \left[\bigcup_{k = -\lfloor N/2\rfloor}^{\lceil N/2\rceil-1}(2k)\right] \bigcup \left[\bigcup_{k = 0}^{\lfloor N/2\rfloor-1}(2k+1)\right] \bigcup \left[\bigcup_{k = 0}^{\lceil N/2\rceil-1}(-2k-1)\right],
	$$
	and this is valid for any natural number $N$. From an implementation point of view, if $U(s)$ is regular enough, note that, although we will work with the whole set $k\in\{-N, \ldots, N-1\}$, it is harmless to impose $\hat U(-N)\equiv 0$ (i.e., we take $k\in\{-N+1, \ldots, N-1\}$ in \eqref{e:deltas12sum}), or to consider $(\hat U(-N) + \hat U(N)) / 2$ instead of $\hat U(-N)$, etc.
	
	Taking into account the previous arguments, \eqref{e:deltas12sum} evaluated at $s = s_j = \pi(2j+1)/(2N)$ becomes
	\begin{align*}
		(-\Delta)_s^{1/2}U(s_j) & \approx \sum_{k = -\lfloor N/2\rfloor}^{\lceil N/2\rceil-1}\hat{U}(2k)(-\Delta)_s^{1/2}e^{i2ks_j} + \sum_{k = 0}^{\lfloor N/2\rfloor-1}\hat{U}(2k+1)(-\Delta)_s^{1/2}e^{i(2k+1)s_j}
		\cr
		& \qquad  + \sum_{k = 0}^{\lceil N/2\rceil-1}\hat{U}(-2k-1)(-\Delta)_s^{1/2}e^{-i(2k+1)s_j} = [\textrm{III}] + [\textrm{IV}] + [\textrm{V}].
	\end{align*}
	Moreover, from \eqref{e:fraclaps}, $(-\Delta)_s^{1/2}U(s)$ is periodic of period $\pi$, even if $U(s)$ is not, so we just need to compute $(-\Delta)_s^{1/2}U(s_j)$ for $s_j\in(0,\pi)$, i.e., for $j\in\{0, \ldots, N-1\}$, and, when $s_j\in(\pi,2\pi)$, i.e., $j\in\{N, \ldots, 2N-1\}$, use that $(-\Delta)_s^{1/2}U(s_j) = (-\Delta)_s^{1/2}U(s_j - \pi) = (-\Delta)_s^{1/2}U(s_{j-N})$. Bearing in mind this, $[\textrm{III}]$ is identical to \eqref{e:fraclapeven}, except that we have $\hat U(2k)$ instead of $\hat U(k)$, and it follows again after replacing $k$ with $2k$ in the first case in \eqref{e:tha1}:
	\begin{equation*}
		[\textrm{III}] = \frac2L\sin^2(s_j)\sum_{k = -\lfloor N/2\rfloor}^{\lceil N/2\rceil-1}\left[e^{i\pi k/N}|k|\hat U(2k)\right]e^{i\pi2jk/N}.
	\end{equation*}
	With respect to $[\textrm{IV}]$, and $[\textrm{V}]$, we use \eqref{e:sumeiks} and \eqref{e:sumeiksneg} in Lemma \ref{lemma:fast}, respectively. Let us consider $[\textrm{IV}]$ first; taking $M = \lfloor N/2\rfloor-1$, $a_l = \hat U(2l+1)$ in \eqref{e:sumeiks}, and replacing $l$ by $k$:
	\begin{align*}
		[\textrm{IV}] & = {-\frac{2i}{L\pi}}\sum_{k = 0}^{\lfloor N/2\rfloor-1}\frac{\hat U(2k+1)}{2k+3} - \frac{2i}{L\pi}\Big[\cos(s_j) + \sin^2(s_j)\ln\left(\cot\left(\frac {s_j}2\right)\right)\Big]
		\cr
		& \qquad \times\sum_{k = 0}^{\lfloor N/2\rfloor-1}(2k+1)\hat U(2k+1)e^{i(2k+1)s_j} + \frac{2i}{L\pi}\sum_{k = 0}^{\lfloor N/2\rfloor-1}(b^+\ast c^+)_ke^{i2ks_j},
	\end{align*}
	where
	$$
	(b^+\ast c^+)_k = \sum_{n = 0}^{P}b_n^+c_{k-n}^+,
	$$
	with $P$ being a natural number, such that $P \ge N$ (strictly speaking, it is enough here that $P \ge 2(\lfloor N/2\rfloor-1)+1 = 2\lfloor N/2\rfloor-1$), and $b_k^+$ and $c_k^+$ being given respectively by
	\begin{equation*}
		b_k^+ =
		\left\{\begin{aligned}
			& (8k+4)\hat U(2k+1), & & 0 \le k \le \lfloor N/2\rfloor-1,
			\\
			& 0, & & \lfloor N/2\rfloor \le k \le P - 1,
		\end{aligned}
		\right.
	\end{equation*}
	and
	\begin{equation*}
		c_k^+ =
		\left\{\begin{aligned}
			& \frac13, & & k = 0, \\
			& 0, & & 1 \le k \le P-\lfloor N/2\rfloor, \\
			& \frac{1}{(2(k-P)-3)(2(k-P)-1)(2(k-P)+1)}, & & P - \lfloor N/2\rfloor + 1 \le k \le P-1.
		\end{aligned}\right.
	\end{equation*}
	Likewise, taking $M = \lceil N/2\rceil-1$, $a_l = \hat U(-2l-1)$ in \eqref{e:sumeiksneg}, and replacing again $l$ by $k$, $[\textrm{V}]$ becomes
	\begin{align*}
		[\textrm{V}] & = \frac{2i}{L\pi}\sum_{k = 0}^{\lceil N/2\rceil-1}\frac{\hat U(-2k-1)}{2k+3} + \frac{2i}{L\pi}\Big[\cos(s_j) + \sin^2(s_j)\ln\left(\cot\left(\frac {s_j}2\right)\right)\Big]
		\cr
		& \qquad \times\sum_{k = 0}^{\lceil N/2\rceil-1}(2k+1)\hat U(-2k-1)e^{-i(2k+1)s_j} - \frac{2i}{L\pi}\sum_{k = 0}^{\lceil N/2\rceil-1}(b^-\ast c^-)_ke^{-i2ks_j},
	\end{align*}
	where
	$$
	(b^-\ast c^-)_k = \sum_{n = 0}^{P}b_n^-c_{k-n}^-,
	$$
	with $P$ being again a natural number, such that $P \ge N$ (strictly speaking, it is enough here that $P \ge 2(\lceil N/2\rceil-1)+1 = 2\lceil N/2\rceil-1$), and $b_k^-$ and $c_k^-$ being given respectively by
	\begin{equation*}
		b_k^- =
		\left\{\begin{aligned}
			& (8k+4)\hat U(-2k-1), & & 0 \le k \le \lceil N/2\rceil-1,
			\\
			& 0, & & \lfloor N/2\rfloor \le k \le P - 1,
		\end{aligned}
		\right.
	\end{equation*}
	and
	\begin{equation*}
		c_k^- =
		\left\{\begin{aligned}
			& \frac13, & & k = 0, \\
			& 0, & & 1 \le k \le P - \lceil N/2\rceil, \\
			& \frac{1}{(2(k-P)-3)(2(k-P)-1)(2(k-P)+1)}, & &  P - \lceil N/2\rceil + 1 \le k \le P-1.
		\end{aligned}\right.
	\end{equation*}
	Putting $[\textrm{IV}]$ and $[\textrm{V}]$ together:
	\begin{align*}
		& [\textrm{IV}] + [\textrm{V}] = {-\frac{2i}{L\pi}}\left[\sum_{k = 0}^{\lfloor N/2\rfloor-1}\frac{\hat U2k+1)}{2k+3} - \sum_{k = 0}^{\lceil N/2\rceil-1}\frac{\hat U(-2k-1)}{2k+3} \right]
		\cr
		&  \quad - \frac{2i}{L\pi}\Big[\cos(s_j) + \sin^2(s_j)\ln\left(\cot\left(\frac {s_j}2\right)\right)\Big]
		\cr
		& \quad \ \ \times\left[\sum_{k = 0}^{\lfloor N/2\rfloor-1}(2k+1)\hat U(2k+1)e^{i(2k+1)s_j} - \sum_{k = 0}^{\lceil N/2\rceil-1}(2k+1)\hat U(-2k-1)e^{-i(2k+1)s_j}\right]
		\cr
		& \quad + \frac{2i}{L\pi}\left[\sum_{k = 0}^{\lfloor N/2\rfloor-1}(b^+\ast c^+)_ke^{i2ks_j} - \sum_{k = 0}^{\lceil N/2\rceil-1}(b^-\ast c^-)_ke^{-i2ks_j}\right]
		\cr
		& \quad = [\mathrm{VI}] + [\mathrm{VII}] + [\mathrm{VIII}].
	\end{align*}
	From an implementation point of view, $[\mathrm{VI}]$ can be expressed in a more compact way:
	\begin{equation*}
		[\mathrm{VI}] = {-\frac{2i}{L\pi}}\sum_{k = -\lceil N/2\rceil}^{\lfloor N/2\rfloor-1}\frac{\sgn(2k+1)\hat U(2k+1)}{|2k+1|+2}.
	\end{equation*}
	With respect to $[\mathrm{VII}]$, we rewrite the bracketed expression:
	\begin{align*}
		& \sum_{k = 0}^{\lfloor N/2\rfloor-1}(2k+1)\hat U(2k+1)e^{i(2k+1)s_j} - \sum_{k = 0}^{\lceil N/2\rceil-1}(2k+1)\hat U(-2k-1)e^{-i(2k+1)s_j}
		\cr
		& \qquad = \sum_{k = -\lceil N/2\rceil}^{\lfloor N/2\rfloor-1}(2k+1)\hat U(2k+1)e^{i(2k+1)\pi(2j+1)/(2N)}.
		\cr
		& \qquad = e^{i\pi(2j+1)/(2N)}\sum_{k = 0}^{\lfloor N/2\rfloor-1}\left[e^{ik\pi/N}(2k+1)\hat U(2k+1)\right]e^{i\pi2jk/N}
		\cr
		& \qquad \quad + e^{i\pi(2j+1)/(2N)}\sum_{k = \lfloor N/2\rfloor}^{N-1}\left[e^{i(k-N)\pi/N}(2(k-N)+1)\hat U(2(k-N)+1)\right]e^{i\pi2jk/N},
	\end{align*}
	i.e., we have the inverse discrete Fourier transform of
	\begin{align*}
		\Big\{N\hat U(1), & 3Ne^{i\pi/N}\hat{U}(3), \ldots, (2\lfloor N/2\rfloor-1)Ne^{i(\lfloor N/2\rfloor-1)\pi/N}\hat{U}(2\lfloor N/2\rfloor-1),
		\cr
		& (-2\lceil N/2\rceil + 1)Ne^{-i\lceil N/2\rceil\pi/N}\hat{U}(-2\lceil N/2\rceil + 1), \ldots, -Ne^{-i\pi/N}\hat{U}(-1)\Big\},
	\end{align*}
	and multiply the result by $e^{is_j} = e^{i\pi(2j+1)/(2N)}$, for $j \in\{0, \ldots, N-1\}$, and then by
	$$
	- \frac{2i}{L\pi}\Big[\cos(s_j) + \sin^2(s_j)\ln\left(\cot\left(\frac {s_j}2\right)\right)\Big].
	$$
	Finally, with respect to $[\mathrm{VIII}]$,
	\begin{align*}
		[\mathrm{VIII}] & = \frac{2i}{L\pi}\left[(b^+\ast c^+)_0 - (b^-\ast c^-)_0\right] + \frac{2i}{L\pi}\sum_{k = 1}^{\lfloor N/2\rfloor-1}(b^+\ast c^+)_ke^{i2ks_j}
		\cr
		& \qquad  - \frac{2i}{L\pi}\sum_{k = 1}^{\lceil N/2\rceil-1}(b^-\ast c^-)_ke^{-i2ks_j}
		\cr
		& = \frac{2i}{L\pi}\left[(b^+\ast c^+)_0 - (b^-\ast c^-)_0\right] + \frac{2i}{L\pi}\sum_{k = 1}^{\lfloor N/2\rfloor-1}\left[e^{ik\pi/N}(b^+\ast c^+)_k\right]e^{i\pi2jk/N}
		\cr
		& \qquad - \frac{2i}{L\pi}\sum_{k = \lfloor N/2\rfloor+1}^{N-1}\left[e^{i(k-N)\pi/N}(b^-\ast c^-)_{N-k}\right]e^{i\pi2jk/N},
	\end{align*}
	i.e., we compute the inverse discrete Fourier transform of
	\begin{align*}
		& \Big\{N\left[(b^+\ast c^+)_0 - (b^-\ast c^-)_0\right],
		\cr
		& \qquad Ne^{i\pi/N}(b^+\ast c^+)_1, \ldots, Ne^{i(\lfloor N/2\rfloor-1)\pi/N}(b^+\ast c^+)_{\lfloor N/2\rfloor-1}, 0,
		\cr
		& \qquad -Ne^{i(\lceil N/2\rceil-1)\pi/N}(b^+\ast c^-)_{\lceil N/2\rceil-1}, \ldots, -Ne^{i\pi/N}(b^+\ast c^-)_1\Big\},
	\end{align*}
	and multiply the result by $2i / (L\pi)$.
	
	From an implementational point of view in Matlab, we have taken $P = N$ in the computations of $(b^+\ast c^+)_k$ and $(b^-\ast c^-)_k$. We have defined different sets of frequencies, namely \verb"k", \verb"kodd", \verb"ka", \verb"kb", \verb"koddposa" and \verb"koddposb", which results in a more compact code. Moreover, we have omitted the division by $2N$ when computing \verb"u2_" and the final multiplication by $2N$ in \verb"Fu_num", because the two factors cancel each other. Except from that, the code matches carefully the steps explained in this section.
	
\begin{lstlisting}[label=code:nonperiodic,language=Matlab, basicstyle={\footnotesize\ttfamily}, caption = {Half Laplacian of $u(x)$, when $U(s)$ is not necessarily periodic of period $\pi$}]
clear
tic
N=10000019; % number of nodes
L=1.1; % scale factor
s=pi*(1:2:2*N-1)/(2*N); % s_j\in(0,\pi)
x=L*cot(s); % x_j\in(-\infty,\infty)
u=1./(1+x.^4); % u(x_j)
Fu=((1-x.^2).*(x.^4+4*x.^2+1)./(1+x.^4).^2)/sqrt(2); % (-\Delta)^{1/2}u(x_j)
% Define different sets of frequencies, reoredered to work with fft:
k=[0:N-1,-N:-1]; % integers in [-N,N-1]
kodd=k(2:2:2*N); % odd integers in [-N,N-1]
ka=[0:ceil(N/2)-1,-floor(N/2):-1]; % integers in [-floor(N/2),ceil(N/2)-1]
kb=[0:floor(N/2)-1,-ceil(N/2):-1]; % integers in [-ceil(N/2),floor(N/2)-1]
koddposa=1:2:2*ceil(N/2)-1; % odd positive integers in [1,2*ceil(N/2)-1]
koddposb=1:2:2*floor(N/2)-1; % odd positive integers in [1,2*floor(N/2)-1]
u2=[u,u(end:-1:1)]; % extend U(s) to s\in[0,2\pi]
u2_=fft(u2); % compute the FFT of u2
u2_(abs(u2_)/(2*N)<eps)=0; % apply Krasny's filter
u2_=exp(-1i*k*pi/(2*N)).*u2_; % this is required, because s_0\not0
% Compute the contribution of the even frequencies to Fu_num:
even=sin(s).^2.*ifft(exp(1i*pi*ka/N).*abs(ka).*u2_(1:2:2*N-1));
% Compute the contribution of the odd frequencies to Fu_num:
odd1=sum(sign(kodd).*u2_(2:2:2*N)./(abs(kodd)+2))/N; % [VI] in the paper
odd2=(cos(s)+sin(s).^2.*log(cot(s/2))).*exp(1i*s) ...
        .*ifft(exp(1i*kb*pi/N).*kodd.*u2_(2:2:2*N)); % [VII] in the paper
bplus=zeros(1,N); % b_k^+
bplus(1:floor(N/2))=4*koddposb.*u2_(2:2:2*floor(N/2));
cplus=zeros(1,N); % c_k^+
cplus([1 N:-1:ceil(N/2)+2])=-1./(koddposb.^3-4*koddposb);
bcplus=ifft(fft(bplus).*fft(cplus)); % (b^+\ast c^+)_k
bminus=zeros(1,N); % b_k^-
bminus(1:ceil(N/2))=4*koddposa.*u2_(2*N:-2:2*N-2*ceil(N/2)+2);
cminus=zeros(1,N); % c_k^-
cminus([1 N:-1:floor(N/2)+2])=-1./(koddposa.^3-4*koddposa);
bcminus=ifft(fft(bminus).*fft(cminus)); % (b^-\ast c^-)_k
odd3=ifft(exp(1i*pi*kb/N).*[bcplus(1)-bcminus(1),bcplus(2:floor(N/2)), ...
        0,-bcminus(ceil(N/2):-1:2)]); % [VIII] in the paper
Fu_num=(even+(1i/pi)*(-odd1-odd2+odd3))/L; % numerical approximation of Fu
if isreal(u), Fu_num=real(Fu_num); end % if u is real, so is Fu_num
norm(Fu_num-Fu,"inf") % error in discrete $L^\infty$-norm
toc % ellapsed time
\end{lstlisting}\label{p:applapfrac}
	
	To make a fair comparison, even if this section is devoted to the case when $U(s)$ is not periodic of period $\pi$, we have consider the same function and parameters as in the previous example, namely $u(x) = (1 + x^4)^{-1}$, $L=1.1$ and $N = 10000019$ (in the next section, we will test the code also with functions $U(s)$ that are not periodic of period $\pi$). The code takes now $17.82$ seconds to execute, and the error is $1.6986\times10^{-14}$. Furthermore, we have slightly adapted the code, as explained at the end of Section \ref{s:periodic}, in order to measure the time required by \verb|fft| and \verb|ifft| , which is of $14.65$ seconds, i.e., approximately $82\%$ of the total time, so the cost of the non-FFT part is of only $3.17$ seconds. On the other hand, when  $N = 16777216$, it takes only $7.88$ seconds to execute, and the error is $1.5543\times10^{-14}$. Moreover, in this case, the calls to \verb|fft| and \verb|ifft| require $3.68$ seconds, i.e., approximately $47\%$ of the total time, and the non-FFT parts $4.20$ seconds. Therefore, the results show again that the accuracy does not degrade when $N$ is extremely large, but the total cost is largely determined by the FFT.
	
	\section{Numerical experiments}\label{s:experiments}
	
	In order to test our method, we have approximated numerically the half Laplacian of functions having different types of decay as $x\to\pm\infty$ and for different values of $N$ and $L$. We measure the accuracy of the results by means of the discrete $L^\infty$-norm of the diference between $(-\Delta)^{1/2}u(x)$ and its respective numerical approximation, which we denote $(-\Delta)_{num}^{1/2}u(x)$:
	\begin{align*}
		\left\|(-\Delta)^{1/2}u(x) - (-\Delta)_{num}^{1/2}u(x)\right\|_\infty & = \max_{j}\left|(-\Delta)^{1/2}u(x_j) - (-\Delta)_{num}^{1/2}u(x_j)\right|
		\cr
		& = \max_{j}\left|(-\Delta)^{1/2}_sU(s_j) - (-\Delta)_{s,num}^{1/2}U(s_j)\right|.
	\end{align*}
	With respect to $L$, let us remark that its choice is a delicate one, because, as said in, e.g., \cite{cayamacuestadelahoz2021}, although there are some theoretical results \cite{Boyd1982}, an adequate election of $L$ does depend on many factors, such as the number of points, class of functions, type of problem, etc. However, a good working rule of thumb seems to be that the absolute value of a given function at the extreme grid points is smaller than a given threshold.
	
	\subsection{Examples with $U(s)$ periodic of period $\pi$}
	
	\label{s:numericalperiodic}
	
	As we have commented in Section~\ref{s:periodic}, the implementation of this case can be done in a straightforward way, by just considering $e^{iks}$ with $k$ even. However, extending the domain of $U(s)$ to $[0, 2\pi]$ may be useful in some cases, depending on how $U(s)$ is defined on $[\pi, 2\pi]$ (e.g., by using an odd or an even extension at $s = \pi$). Therefore, we have considered $s\in[0,2\pi]$ in all the numerical experiments in this section. On the other hand, in this paper, unless otherwise indicated, we plot in semilogarithmic scale the $L^\infty$-norm of the errors, with an abscissa range between $10^{-17}$ and $10$, which facilitates the comparison of different numerical experiments.
	
	\subsubsection{$u(x) = (1 + x^4)^{-1}$}
	
	This is the function that we have considered in the Matlab codes in Section \ref{s:implementation}; recall that its half Laplacian is given by \eqref{e:fraclap11x4}. We have considered $s\in[0,2\pi]$, extending $U(s)$ in two different ways: an even extension at $s = \pi$, which is done by typing \verb"u2=[u,u(end:-1:1)];", as in the Matlab code in Section \ref{s:implementationnotperiodic}, and an odd extension at $s = \pi$, for which we type \verb"u2=[u,-u(end:-1:1)];". Note that $U(s) = (1 + L^4\cot^4(s))^{-1}$, so it is straightforward to check that an even extension yields $U(s)\in\mathcal C^\infty(\mathbb R)$, whereas an odd extension yields $U(s)\in\mathcal C^3(\mathbb R)$.
	
	In Figure \ref{f:fraclap11x4}, we show the errors for $L \in\{0.01, 0.02, \ldots, 10\}$ and $N \in \{2^2, 2^3, \ldots, 2^{13}\}$; the results on the left-hand side correspond to an even extension, and those on the right-hand side, to an odd one. Note that, for this function, an even extension of $U(s)$ at $s = \pi$ is also periodic of period $\pi$, so considering $U(s)$ only on $[0,\pi]$ yields exactly the same results up to infinitesimal rounding errors, and the resulting graphic is visually indistinguishable from the left-hand side of Figure \ref{f:fraclap11x4}.
	\begin{figure}[!htbp]
		\centering
		\includegraphics[width=0.5\textwidth, clip=true]{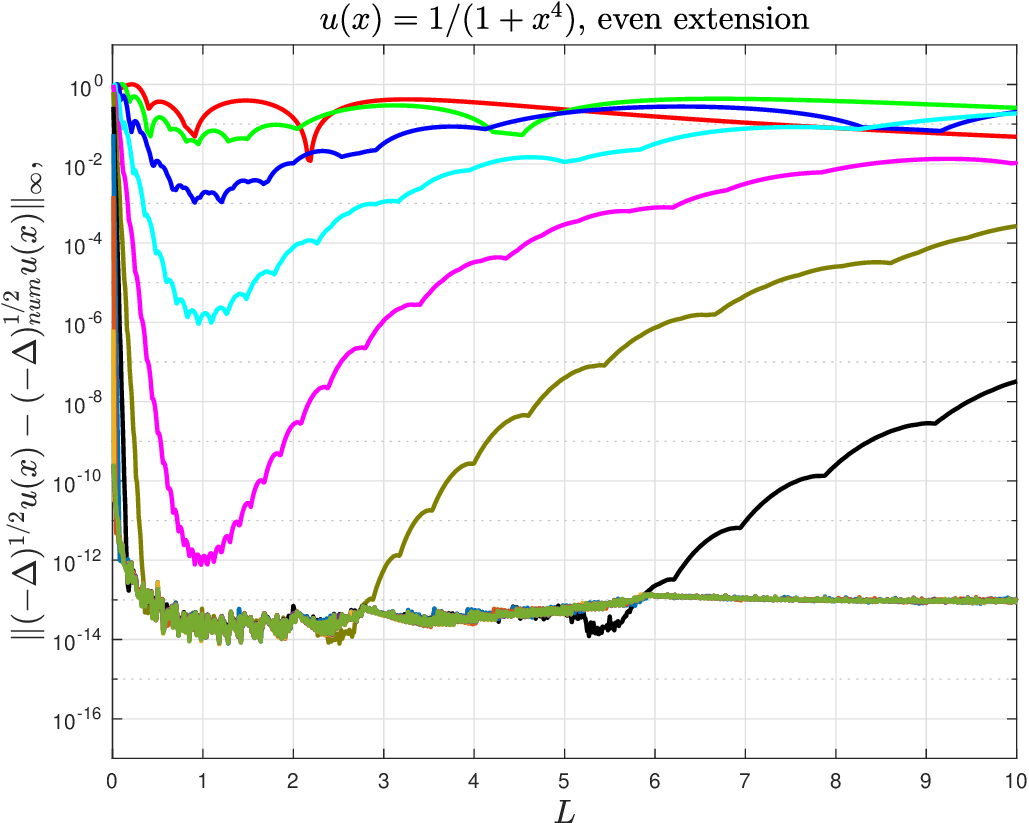}\includegraphics[width=0.5\textwidth, clip=true]{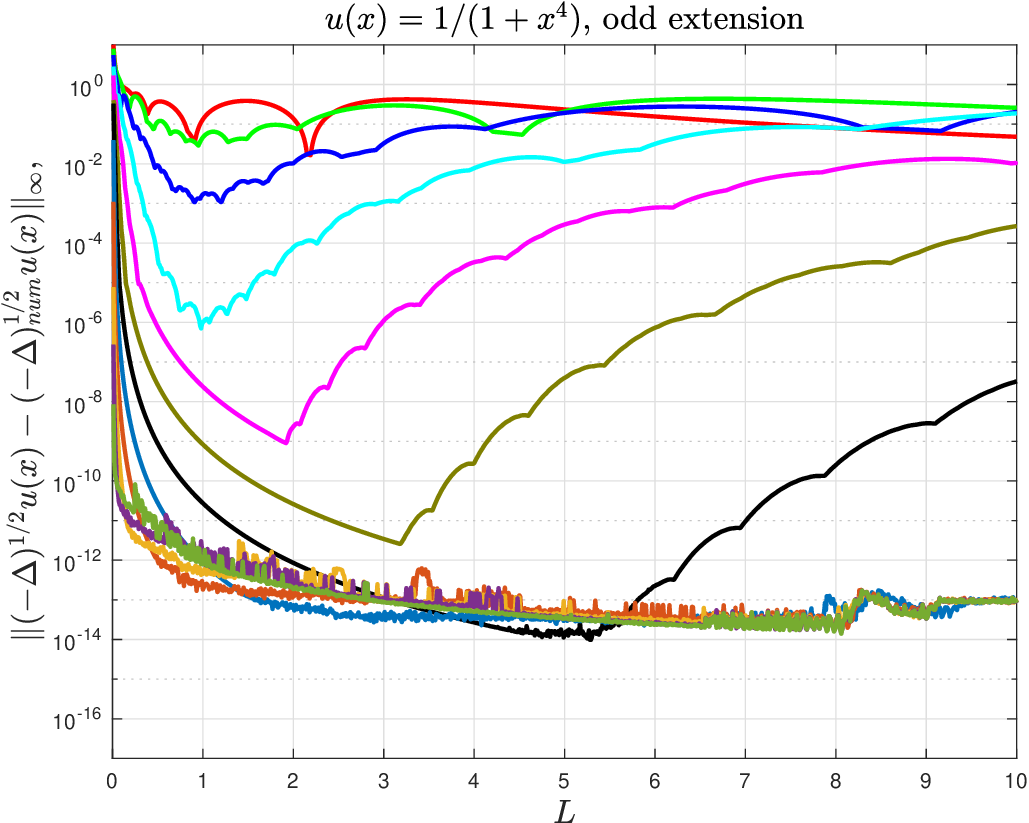}
		\includegraphics[width=\textwidth, clip=true]{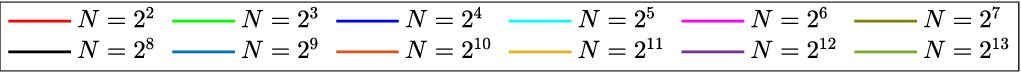}
		\caption{Errors in the numerical approximation of $(-\Delta)^{1/2}u(x)$, for $u(x) = (1+x^4)^{-1}$, taking $N \in\{2^2, 2^3, \ldots, 2^{13}\}$ and $L\in\{0.01, 0.02, \ldots, 10\}$. Left: we have considered an even extension at $s = \pi$. Right: we have considered an odd extension at $s = \pi$.}
		\label{f:fraclap11x4}
	\end{figure}
	The numerical results reveal that, although for some values of $N$ (especially $N = 64$ and $N = 128$) the even extension works better, for larger $N$ there are no remarkable differences, and the highest possible precision can be reached in both cases for a large range of values of $L$.
	
	\subsubsection{$u(x) = (1 + x^2)^{-1/2}$} This example clearly illustrates why it can be advantageous to work with $s\in[0,2\pi]$, even when $u \in L^2(\mathbb R)$ (and, hence, $U(s)$ is periodic of period $\pi$).
	
	In order to compute $(-\Delta)^{1/2}(1 + x^2)^{-1/2}$, we make $k = 1$ in \eqref{e:thc} and take the imaginary part of the resulting expression, getting
	\begin{equation*}
		(-\Delta)^{1/2} \frac{1}{(1 + x^2)^{1/2}} = \frac{2(1 + x^2)^{1/2} - 2x\arg\sinh(x)}{\pi(1 + x^2)^{3/2}}.
	\end{equation*}
	As in the previous example, we have considered both an even and an odd extension for $U(s) = (1 + L^2\cot^2(s))^{-1/2}$ at $s = \pi$. In this case, an even extension yields only a globally continuous function $U(s)\in\mathcal C^0(\mathbb R)$, whereas an odd extension yields $U(s)\in\mathcal C^\infty(\mathbb R)$. Note also that the even extension of $U(s)$ at $s = \pi$ is again also periodic of period $\pi$, so considering $U(s)$ only on $[0,\pi]$ yields exactly the same results. In Figure \ref{f:fraclap1sqrt1x2}, we show the errors for $N \in\{2^2, \ldots, 2^{13}\}$. The left-hand side corresponds to the even extension, where we have taken $L \in\{0.01, 0.02, \ldots, 400\}$, and the right-hand side to the odd one, where we have taken $L \in\{0.01, 0.02, \ldots, 10\}$. From the results, we can see that, even if it is possible to reach the highest accuracy with both extensions, the even extension is much less adequate here, because it requires much larger values of $N$ and of $L$, and the optimum value of $L$ (which grows with $N$) must be chosen quite carefully. On the other hand, when $L = 1$, an odd extension at $s = \pi$ corresponds exactly to $\sin(s) = (e^{is} - e^{-is}) / 2i$, so the numerical approximation is exact for all $N$, up to infinitesimally small rounding errors, but, when $L \not=1$, an odd extension is also convenient, and the best results can be indeed achieved for a large range of values of $L$.
	\begin{figure}[!htbp]
		\centering
		\includegraphics[width=0.5\textwidth, clip=true]{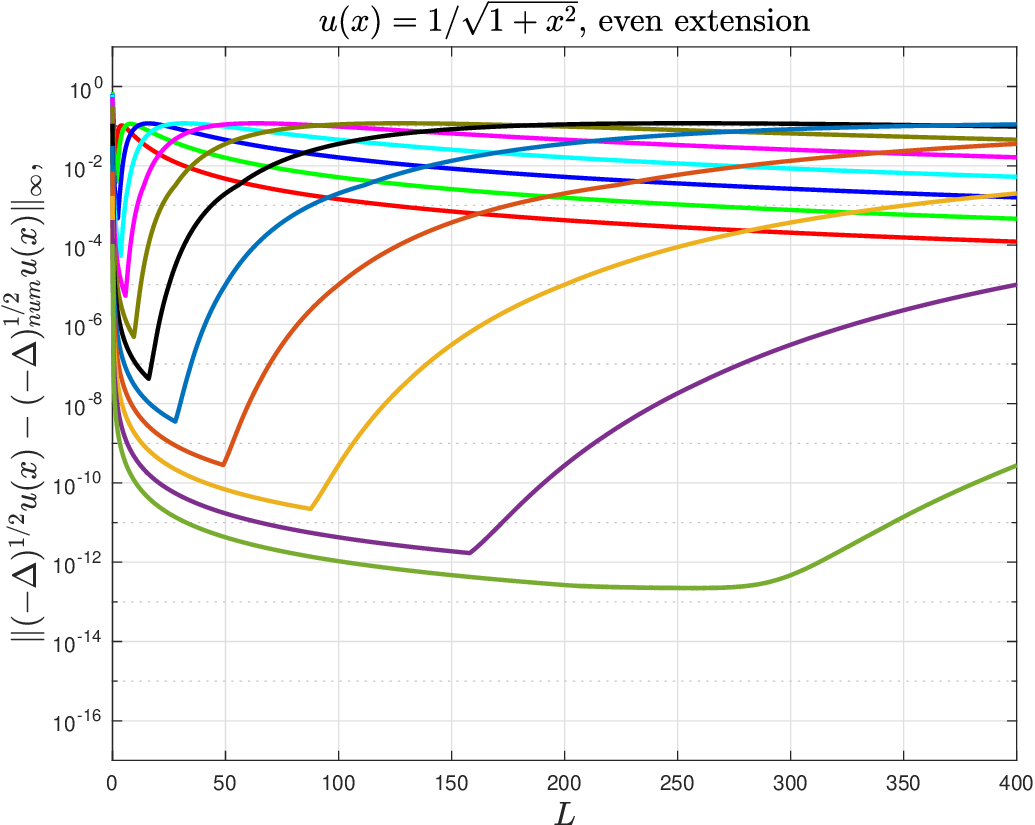}\includegraphics[width=0.5\textwidth, clip=true]{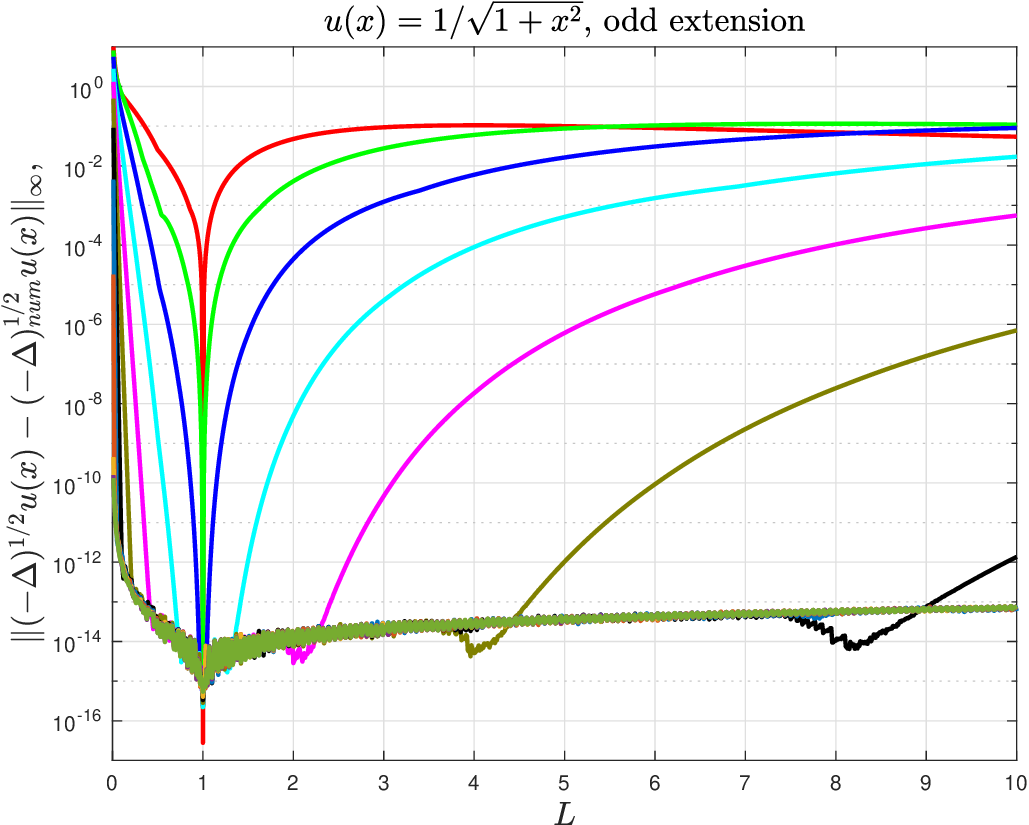}
		\includegraphics[width=\textwidth, clip=true]{legendnumberofpointsN.eps}
		\caption{Errors in the numerical approximation of $(-\Delta)^{1/2}u(x)$, for $u(x) = (1+x^2)^{-1/2}$, taking $N \in\{2^2, 2^3, \ldots, 2^{13}\}$. Left: we have considered an even extension at $s = \pi$, and taken $L\in\{0.01, 0.02, \ldots, 400\}$. The graphic corresponding to $U(s)$ with $s\in[0,\pi]$ and no extension is identical. Right: we have considered an odd extension at $s = \pi$, and taken $L\in\{0.01, 0.02, \ldots, 10\}$.}
		\label{f:fraclap1sqrt1x2}
	\end{figure}
	
	In general, if $U(s)\in\mathcal C^0([0,\pi])$, an even extension at $s = \pi$ yields a periodic function on $s\in[0,2\pi]$ that is globally at least $\mathcal C^0(\mathbb R)$. Then, it is well known that if a periodic function $U(s)\in\mathcal C^r(\mathbb R)$, then its Fourier coefficients, for a fixed $L$, decay asymptotically as powers of $r + 2$. For instance, in the even extension in this example, the Fourier coefficients decay asymptotically in a quadratic way, so, for a fixed $L$, we can expect that the errors decay as $\mathcal O(N^{-2})$, for $N$ large enough. To illustrate this, in Figure \ref{f:fraclap1sqrt1x2LN}, we have plotted the decimal logarithm of $N$ versus the decimal logarithm of the errors corresponding to the even extension at $s = \pi$, for $N \in\{2, 3, \ldots, 10^4\}$ and $L \in\{2^{-1}, 2^0, \ldots, 2^{7}\}$. For a given $L$, which has its associated color, the corresponding errors form a curve that tends, as $N$ increases, to a straight line of slope approximately equal to $-2$.
	
	\begin{figure}[!htbp]
		\centering
		\includegraphics[width=0.5\textwidth, clip=true]{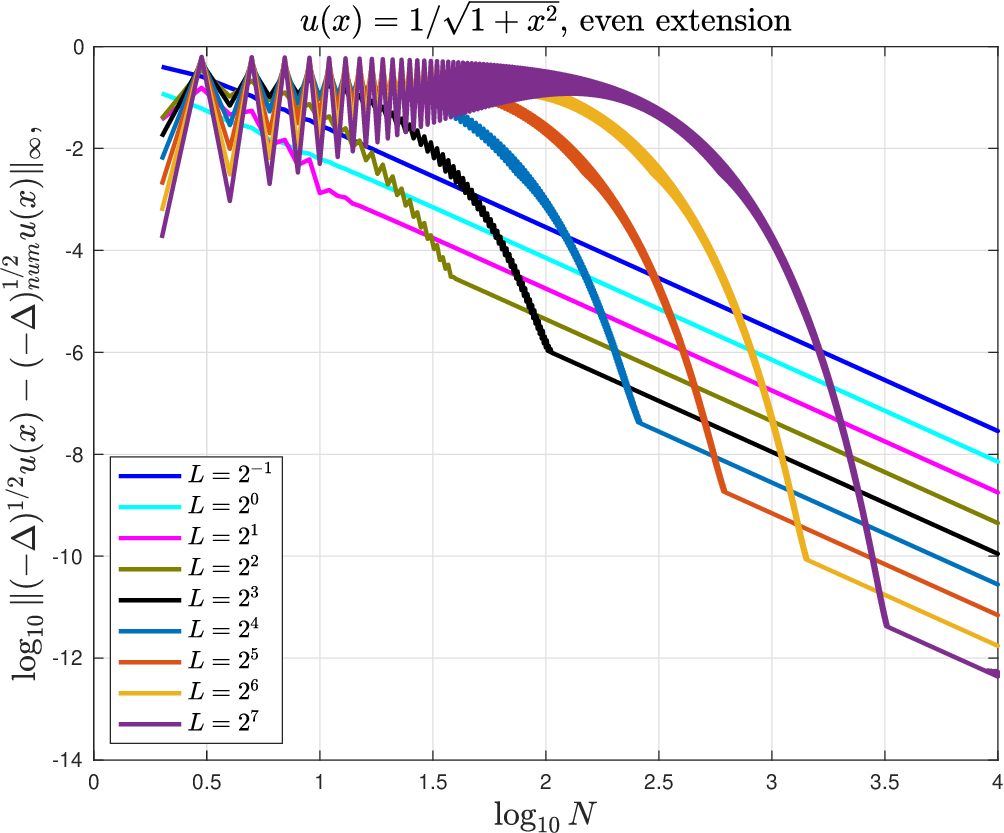}
		\caption{Errors in the numerical approximation of $(-\Delta)^{1/2}u(x)$, for $u(x) = (1+x^2)^{-1/2}$, considering an even extension at $s = \pi$, and taking $N \in\{2, 3, \ldots, 10^4\}$ and $L \in\{2^{-1}, 2^0, \ldots, 2^7\}$.}
		\label{f:fraclap1sqrt1x2LN}
	\end{figure}

	\subsection{Examples with $U(s)$ not periodic of period $\pi$}
	
	\label{s:numericalnotperiodic}
	
	In this case, it is necessary to extend $U(s)$ to $s\in[0, 2\pi]$, although we will study in Section \ref{s:nonperiodictoperiodic} some strategies to reduce the nonperiodic case to the periodic one.
	
	\subsubsection{$u(x) = \arctan(x)$} \label{s:atan} We know (see, e.g., \cite{weideman1995}) that 
	$$
	\mathcal H\left(\frac{1}{1+x^2}\right) = \frac{x}{1+x^2},
	$$
	so, bearing in mind that $\arctan'(x) = (1 + x^2)^{-1}$, it follows that
	\begin{align}
		\label{e:fraclapatan}
		(-\Delta)^{1/2}\arctan(x) = \frac{x}{1 + x^2}.
	\end{align}
	On the left-hand side of Figure \ref{f:fraclapatan}, we have considered an even extension of $U(s) = \arctan(L\cot(s))$ at $s = \pi$, taking $L\in\{0.01, 0.02, \ldots, 400\}$ and $N \in\{2^2, 2^3, \ldots, 2^{13}\}$. The resulting graphic is not identical to, but closely resembles the left-hand side of Figure \ref{f:fraclap1sqrt1x2}, because, again, $U(s)\in\mathcal C^0(\mathbb R)$. As can be seen, it is certainly possible to achieve the highest accuracy, although rather large values of $N$ and $L$ are required. Therefore, even if the even extension at $s = \pi$ still works, we can regard it as not ideal. Indeed, for instance, when $L = 1$, $\arctan(x) = \arctan(\cot(s)) = \pi/2 - s$, so we are indeed computing the Fourier expansion of $U(s) = |\pi - s| - \pi / 2\in\mathcal C^0(\mathbb R)$, whose Fourier coefficients decay only quadratically.
	
	\begin{figure}[!htbp]
		\centering
		\includegraphics[width=0.5\textwidth, clip=true]{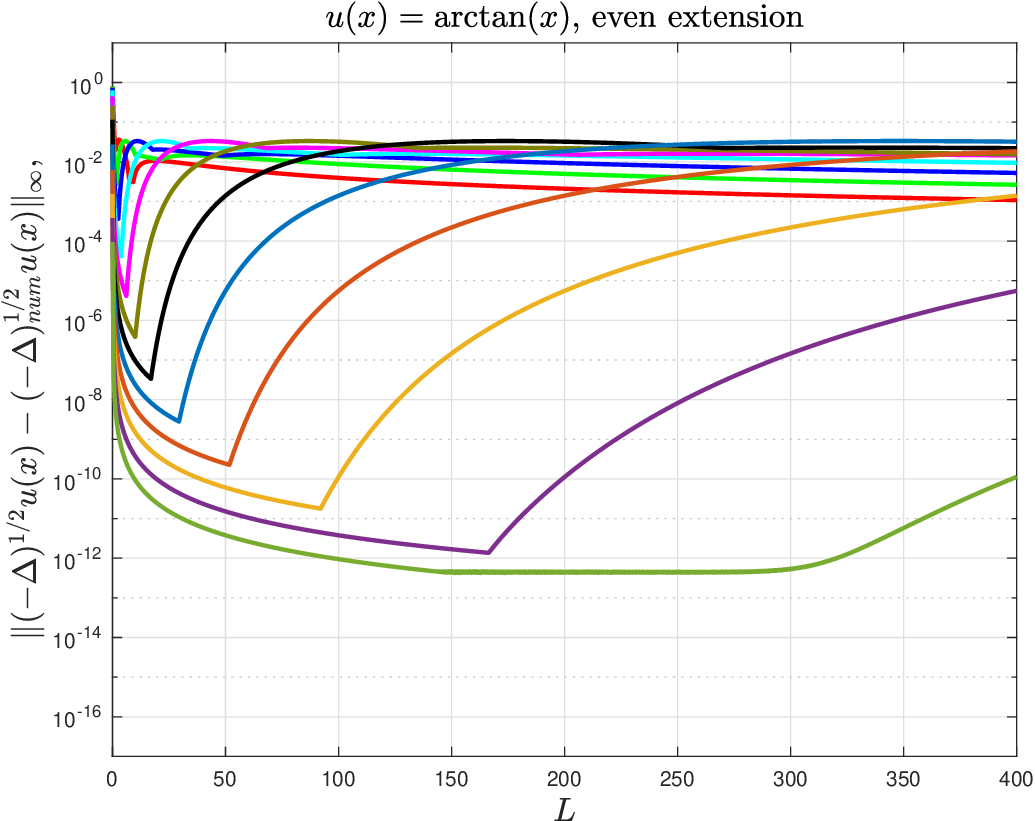}\includegraphics[width=0.5\textwidth, clip=true]{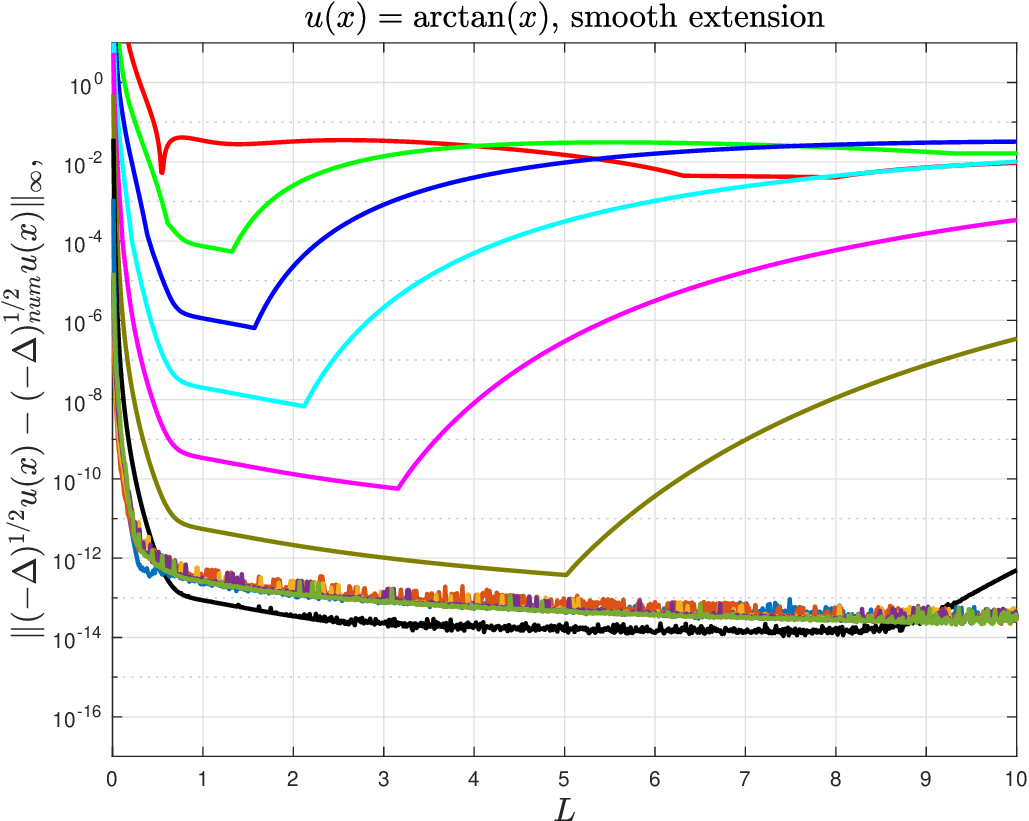}
		\includegraphics[width=\textwidth, clip=true]{legendnumberofpointsN.eps}
		\caption{Errors in the numerical approximation of $(-\Delta)^{1/2}u(x)$, for $u(x) = \arctan(x)$, taking $N \in\{2^2, 2^3, \ldots, 2^{13}\}$. Left: We have considered an even extension of $U(s)$ at $s = \pi$, and taken $L\in\{0.01, 0.02, \ldots, 400\}$. Right: we have considered the smooth extension given by \eqref{e:usmooth} for $s\in[\pi,2\pi]$, and taken $L\in\{0.01, 0.02, \ldots, 10\}$.}
		\label{f:fraclapatan}
	\end{figure}
	
	In order to improve the results, we have constructed a smooth extension from $s\in[0,\pi]$ to $s\in[0,2\pi]$. A detailed study on how to extend functions defined over half a period to the whole period lies beyond the scope of this paper. In our case, even if this can be done in different ways (see, e.g., \cite{Boyd2002,Boyd2006,Huybrechs2010,mortonsilverberg2009}), we have found that a simple trigonometric interpolant is enough. More precisely, we propose the following ansatz:
	\begin{equation*}
		U(s) =
		\left\{
		\begin{aligned}
			& \arctan(L\cot(s)), & & s\in[0,\pi]
			\cr
			& \sum_{k = 1}^{5}[\alpha_k\cos(ks) + \beta_k\sin(ks)], & & s\in(\pi,2\pi],
		\end{aligned}
		\right.
	\end{equation*}
	and the other coefficients $\alpha_k$, $\beta_k$ are determined in such a way that $U(s)$ is at least of class $\mathcal C^4(\mathbb R)$, which is achieved by imposing the following conditions:
	\begin{equation}
		\label{e:triginterp}
		\begin{aligned}
			U(\pi) & = -\frac{\pi}{2}, & U(2\pi) & \equiv U(0) = \frac{\pi}{2},
			\cr
			U'(\pi) & = -\frac{1}{L}, & U'(2\pi) & \equiv U'(0) = -\frac{1}{L},
			\cr
			U''(\pi) & = 0, & U''(2\pi) & \equiv U''(0) = 0,
			\cr
			U'''(\pi) & = -\frac2L + \frac2{L^3}, & U'''(2\pi) & \equiv U'''(0) = -\frac2L + \frac2{L^3},
			\cr
			U^{(4)}(\pi) & = 0, & U^{(4)}(2\pi) & \equiv U^{(4)}(0) = 0.
		\end{aligned}
	\end{equation}
	Then, it is not difficult to conclude that
	\begin{equation}
		\label{e:usmooth}
		U(s) =
		\left\{
		\begin{aligned}
			& \arctan(L\cot(s)), & & s\in[0,\pi]
			\cr
			& \frac{75\pi}{128}\cos(s) + \left(-\frac{3}{4L} + \frac{1}{12L^3}\right)\sin(2s) 
			\cr 
			& \quad - \frac{25\pi}{256}\cos(3s) + \left(\frac{1}{8L} - \frac{1}{24L^3}\right)\sin(4s) + \frac{3\pi}{256}\cos(5s), & & s\in(\pi,2\pi].
		\end{aligned}
		\right.
	\end{equation}
	Moreover, for all $m\in\mathbb N$,  $U^{(m)}(0^+) = U^{(m)}(\pi^-)$, but $U^{(5)}(0^+) = U^{(5)}(\pi^-) =  -16/L + 40/L^3 - 24/L^5$ and $U^{(5)}(\pi^+) = U^{(5)}(2\pi^-) =  104/L - 40/L^3$, so $U(s)\in\mathcal C^4(\mathbb R)$.
	
	On the right-hand side of Figure \ref{f:fraclapatan}, we have considered the smooth extension given by \eqref{e:usmooth}, taking $L\in\{0.01, 0.02, \ldots, 10\}$ and $N \in\{2^2, 2^3, \ldots, 2^{13}\}$. The results show that, except in the case when $L$ is very small (where the results are completely erroneous), it is now possible to achieve an accuracy of the order of $\mathcal O(10^{-13})$ with just $N = 128$ points, and, in general, the highest accuracy is reached for a large range of values of $L$.
	
	\subsubsection{$u(x) = x(1+x^2)^{-1/2}$}
	In order to compute $(-\Delta)^{1/2}x(1 + x^2)^{-1/2}$, we make $k = 1$ in \eqref{e:thc} and take the real part of the resulting expression, getting
	\begin{equation}
		\label{e:fraclapxsqrt1x2}
		(-\Delta)^{1/2} \frac{x}{(1+x^2)^{1/2}} = \frac{2x(1+x^2)^{1/2} + 2\arg\sinh(x)}{\pi(1+x^2)^{3/2}}.
	\end{equation}
	In Figure \ref{f:fraclapxsqrt1x2}, we have considered an even extension at $s = \pi$, taking $L\in\{0.01, 0.02, \ldots, 10\}$ and $N \in\{2^2, 2^3, \ldots, 2^{13}\}$. As happened in the odd extension on the right-hand side of Figure \ref{f:fraclap1sqrt1x2}, when $L = 1$, an even extension at $s = \pi$ corresponds exactly to an elementary trigonometric function, in this case $\cos(s) = (e^{is} + e^{-is}) / 2$, so the numerical approximation is exact for all $N$, up to infinitesimally small rounding errors. However, even when $L \not=1$, the best results can be indeed achieved for a large range of values of $L$. Indeed, in this case, $U(s) = L\cot(s)(1+L^2\cot^2(s))^{-1/2}$, so a an even extension on $s = \pi$ gives us a periodic function of period $2\pi$ that is also $\mathcal C^\infty(\mathbb R)$.
	
	\begin{figure}[!htbp]
		\centering
		\includegraphics[width=0.5\textwidth, clip=true]{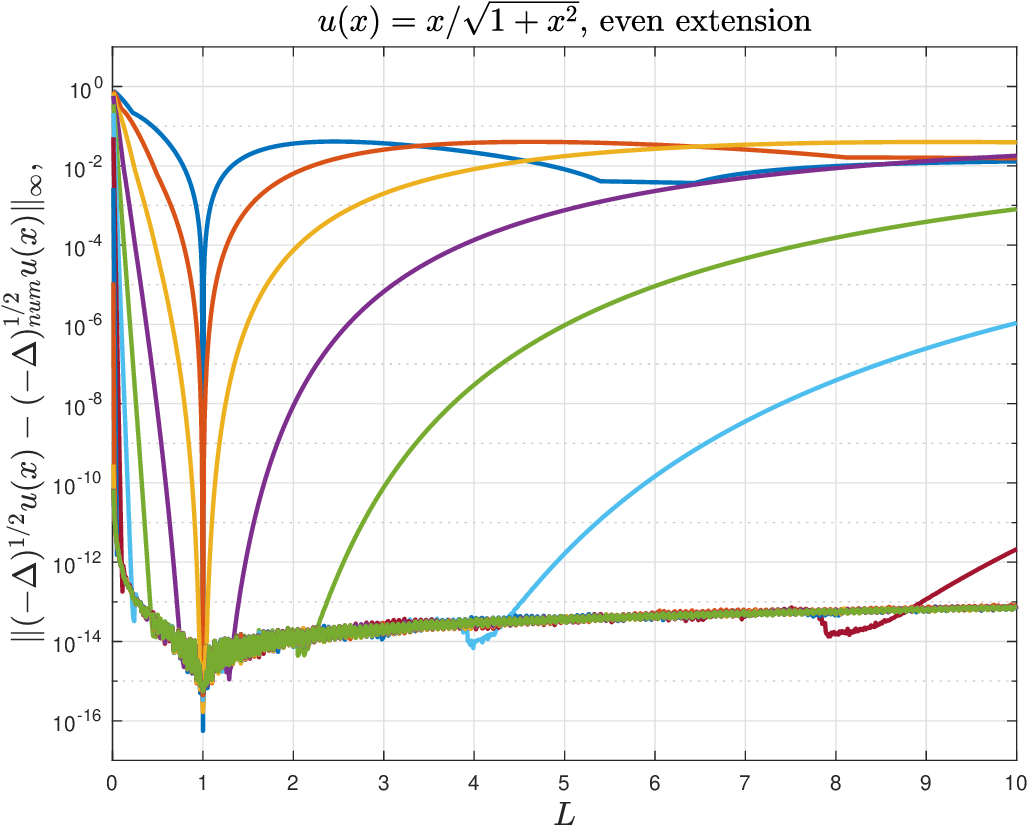}
		\includegraphics[width=\textwidth, clip=true]{legendnumberofpointsN.eps}
		\caption{Errors in the numerical approximation of $(-\Delta)^{1/2}u(x)$, for $u(x) = x(1+x^2)^{-1/2}$, taking $L\in\{0.01, 0.02, \ldots, 10\}$ and $N \in\{2^2, 2^3, \ldots, 2^{13}\}$. We have considered an even extension of $U(s)$ at $s = \pi$.}
		\label{f:fraclapxsqrt1x2}
	\end{figure}
	
	\subsubsection{$u(x) = \erf(x)$}We know (see, e.g., \cite{weideman1995}) that 
	$$
	\mathcal H\left(e^{-x^2}\right) = \frac{2}{\sqrt\pi}e^{-x^2}\int_{0}^{x}e^{t^2}dt,
	$$
	and the $\erf$ function, which can be evaluated numerically by means of, e.g., the Matlab command \verb"erf", is defined as
	$$
	\erf(x) = \frac{2}{\sqrt\pi}\int_0^xe^{-t^2}dt;
	$$
	hence, it follows that
	\begin{align}
		\label{e:delta12erf}
		(-\Delta)^{1/2}\erf(x) = \frac{4}{\pi}e^{-x^2}\int_{0}^{x}e^{t^2}dt = \frac{4}{\pi}D(x),
	\end{align}
	where $D(x)$ is Dawson's integral:
	$$
	D(x) = e^{-x^2}\int_{0}^{x}e^{t^2}dt,
	$$
	which can be evaluated numerically, e.g., as explained in \cite{weideman1995} or \cite{Nijimbere2019}, although, in our case, we have used the Matlab command \verb"dawson". In fact, this example is interesting, because Matlab's implementation of $D(x)$ is remarkably slower than the method proposed in this paper. For instance, when $N = 8192$ and $L = 10$, our code takes approximately $0.01$ seconds to execute, whereas computing the half Laplacian of the $\erf$ function by invoking Matlab's \verb"dawson" function takes around $6.2$ seconds. For the sake of comparison, we have also tried the corresponding Mathematica command \verb"DawsonF". Then, the elapsed time to compute $(4/\pi)D(x)$ for the values of $x$ we are interested in, which is given by the following piece of code, is of approximately $1.72$ seconds, i.e., almost a fourth of Matlab's time.
	\begin{verbatim}
		AbsoluteTiming[L = 10.; n = 8192.; y = (4./Pi) DawsonF[
		L Cot[Pi Table[(2.*k + 1.)/(2. n), {k, 0, n - 1}]]];]
	\end{verbatim}
	
	In order to get Figure \ref{f:fraclap1sqrt1x2}, we have considered an even extension of $U(s) = \erf(L\cot(s))$ at $s = \pi$, taking $L\in\{0.01, 0.02, \ldots, 10\}$ and $N \in\{2^2, 2^3, \ldots, 2^{13}\}$.  Note that, in this case, $U(s)\in\mathcal C^\infty(\mathbb R)$. In the numerical experiments, we remark that $N = 64$ is enough to attain an accuracy of the order of $\mathcal O(10^{-15})$, and, in general, the best results can be achieved for a large range of values of $L$.
	\begin{figure}[!htbp]
		\centering
		\includegraphics[width=0.5\textwidth, clip=true]{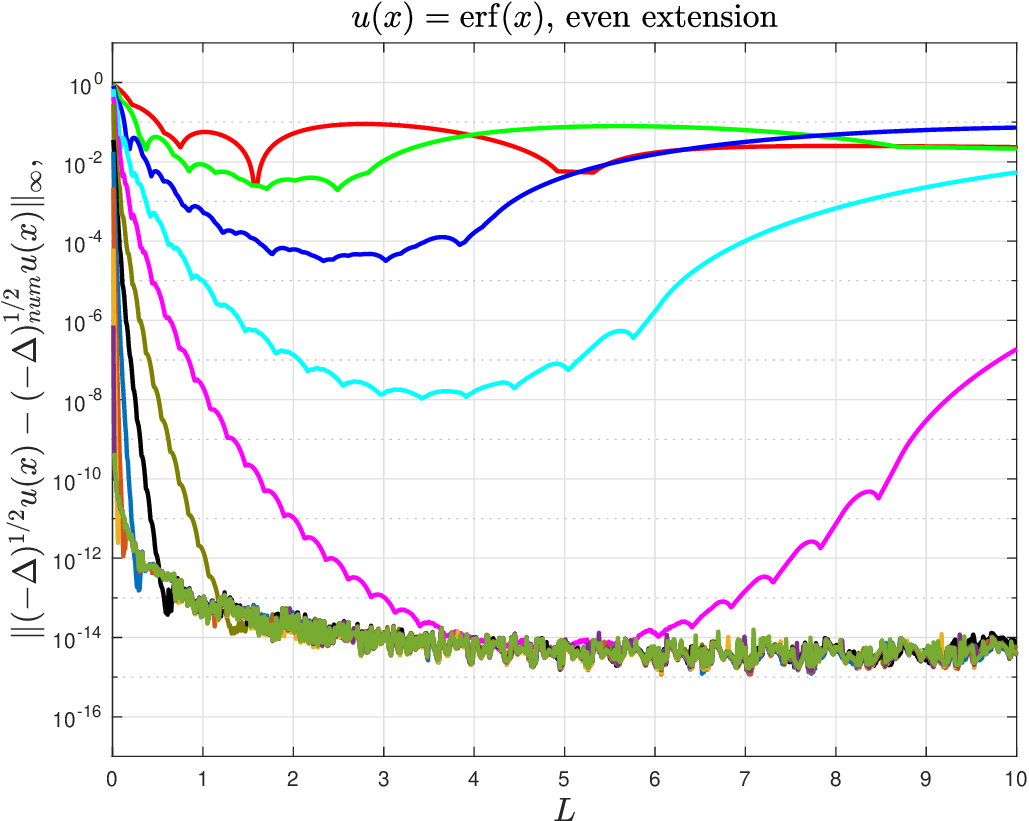}
		\includegraphics[width=\textwidth, clip=true]{legendnumberofpointsN.eps}
		\caption{Errors in the numerical approximation of $(-\Delta)^{1/2}u(x)$, for $u(x) = \erf(x)$, taking $L\in\{0.01, 0.02, \ldots, 10\}$ and $N \in\{2^2, 2^3, \ldots, 2^{13}\}$. We have considered an even extension of $U(s)$ at $s = \pi$.}
		\label{f:fraclaperf}
	\end{figure}
	
	\subsection{Reducing the nonperiodic case to the periodic one} \label{s:nonperiodictoperiodic}
	
	Given a function $u(x)$ such that its corresponding $U(s)$ is not periodic on $s\in[0,\pi]$, a natural question that arises is whether it might be advantageous to find a function $v(x)$, such that $U(s) - V(s)$, with $V(s) \equiv v(x(s))$, can be regarded as period on $s\in[0,\pi]$ and the half Laplacian of $v(x)$ is explicitly known. Then, we would have immediately
	$$
	(-\Delta)_s^{1/2}U(s) = (-\Delta)_s^{1/2}[U(s) - V(s)] + (-\Delta)_s^{1/2}V(s),
	$$
	and the computation of $(-\Delta)_s^{1/2}[U(s) - V(s)]$ would be done as explained in Section \ref{s:periodic}, which is less involved than extending the period from $s\in[0,\pi]$ to $s\in[0,2\pi]$. In order to assess the adequacy of this approach, we will  consider a few examples in the following pages.
	
	\subsubsection{$u(x) = \arctan(x)$} In Section \ref{s:atan}, we have computed numerically the half Laplacian of $(-\Delta)^{1/2}\arctan(x)$, whose exact expression is given by \eqref{e:fraclapatan}, by means of an even extension and a smooth extension. Therefore, an obvious choice of $V(s)$ is taking the same $2\pi$-periodic trigonometric function used in the smooth extension in \eqref{e:usmooth}, but defined now on $s\in[0,\pi]$ instead of $s\in[\pi,2\pi]$, because, from \eqref{e:triginterp}, it interpolates $U(s) = \arctan(L\cot(s))$ at $s = 0,$ and $s = \pi$:
	\begin{align}
		\label{e:Vsatan}
		V(s) & = \frac{75\pi}{128}\cos(s) + \left(-\frac{3}{4L} + \frac{1}{12L^3}\right)\sin(2s)
		\cr 
		& \qquad - \frac{25\pi}{256}\cos(3s)  + \left(\frac{1}{8L} - \frac{1}{24L^3}\right)\sin(4s) + \frac{3\pi}{256}\cos(5s), \quad s\in[0,\pi].
	\end{align}
	Then, defining $W(s) = U(s) - V(s)$, $W(s)$ satisfies immediately $W(0) = W'(0) = W''(0) = W'''(0) = W^{(4)}(0) = 0$ and $W(\pi) = W'(\pi) = W''(\pi) = W'''(\pi) = W^{(4)}(\pi) = 0$. Moreover, $W^{(5)}(0) = W^{(5)}(\pi) = -120/L + 80/L^3 - 24/L^5$, but $W^{(6)}(0) = 225\pi/2$ and $W^{(6)}(\pi) = -225\pi/2$, so $W(s)\in\mathcal C^5(\mathbb R)$, and we can approximate numerically $(-\Delta)^{1/2}_sW(s)$ as explained in Section \ref{s:periodic}.
	
	\begin{figure}[!htbp]
		\centering
		\includegraphics[width=0.5\textwidth, clip=true]{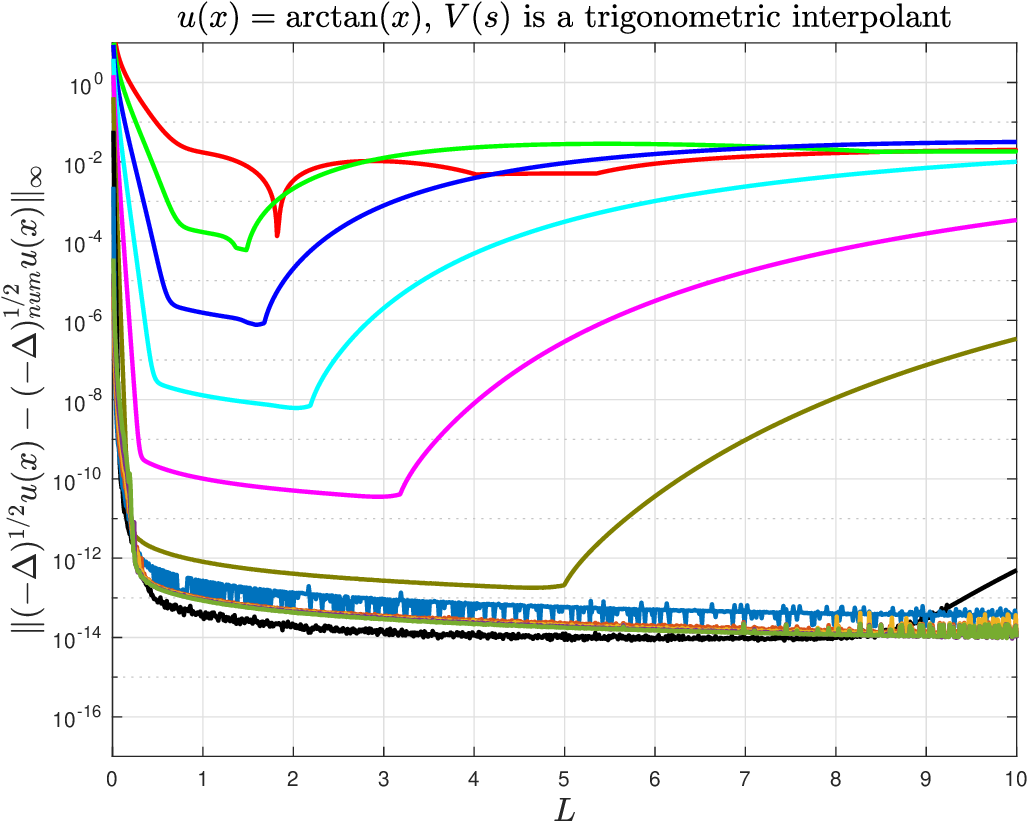}\includegraphics[width=0.5\textwidth, clip=true]{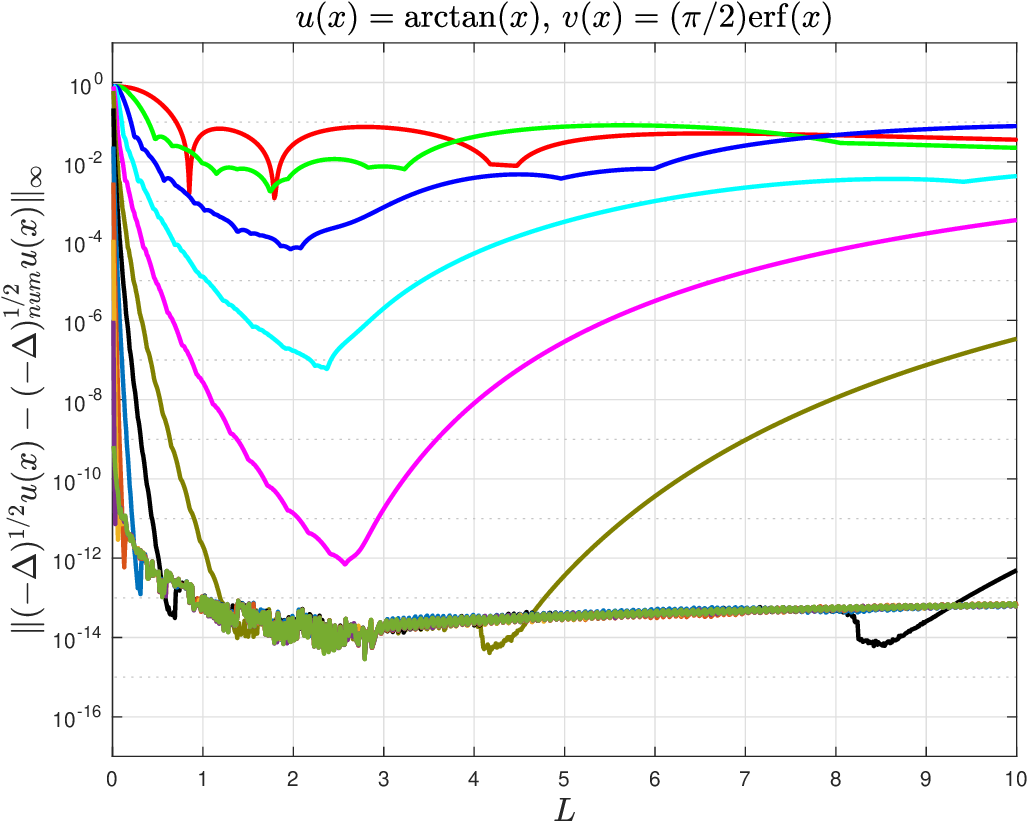}
		\includegraphics[width=\textwidth, clip=true]{legendnumberofpointsN.eps}
		\caption{Errors in the numerical approximation of $(-\Delta)^{1/2}u(x)$, for $u(x) = \arctan(x)$, taking $L\in\{0.01, 0.02, \ldots, 10\}$ and $N \in\{2^2, 2^3, \ldots, 2^{13}\}$. Left: we have considered a trigonometric interpolant as $V(s)$. Right: we have taken $v(x) = \erf(x)$.}
		\label{f:fraclapatanVerf}
	\end{figure}
	
	On the other hand, $(-\Delta)_s^{1/2}U(s) = (-\Delta)_s^{1/2}W(s) + (-\Delta)_s^{1/2}V(s)$, and in our case, the analytic expression of $(-\Delta)_s^{1/2}V(s)$ is known explicitly. Indeed, when $k$ is even, from \eqref{e:tha1},
	$$
	(-\Delta)_s^{1/2}\cos(ks) = \frac{|k|\sin^2(s)}{L}\cos(ks), \qquad (-\Delta)_s^{1/2}\sin(ks) = \frac{|k|\sin^2(s)}{L}\sin(ks),
	$$
	and when $k$ is odd, $(-\Delta)_s^{1/2}\cos(ks)$ and $(-\Delta)_s^{1/2}\sin(ks)$ are given respectively by \eqref{e:deltas12cosks} and \eqref{e:deltas12sinks}. Therefore, introducing the expressions for $\cos(s)$, $\sin(2s)$, $\cos(3s)$, $\sin(4s)$ and $\cos(5s)$ into \eqref{e:Vsatan}, and applying elementary trigonometric identities, we get:
	\begin{align*}
		(-\Delta)_s^{1/2}V(s) & = \frac{15}{512L}(35\sin(s) - 21\sin(3s) + 7\sin(5s) - \sin(7s))\ln\left(\cot\left(\frac s2\right)\right) 
		\cr
		& \qquad + \left(\frac{283}{256L} - \frac{7}{8L^2} + \frac{1}{8L^4}\right)\sin(2s) + \left(-\frac{25}{64L} + \frac{5}{8L^2} - \frac{1}{8L^4}\right)\sin(4s)
		\cr
		& \qquad + \left(\frac{15}{256L} - \frac{1}{8L^2} + \frac{1}{24L^4}\right)\sin(6s).
	\end{align*}
	On the left-hand side of Figure \ref{f:fraclapatanVerf}, we have approximated numerically the half Laplacian of $\arctan(x)$ by adding and subtracting \eqref{e:Vsatan}, taking $L\in\{0.01, 0.02, \ldots, 10\}$ and $N \in\{2^2, 2^3, \ldots, 2^{13}\}$. Although $W(s)$ is globally more regular than  the $U(s)$ constructed in \eqref{e:usmooth} (they are of class $\mathcal C^5(\mathbb R)$ and $\mathcal C^4(\mathbb R)$, respectively), the results are very similar to those on the right-hand side of Figure \ref{f:fraclapatan}, and they improve those on the left-hand side of Figure  \ref{f:fraclapatan} corresponding to an even extension of $U(s)$ at $s = \pi$.
	
	We remark that choosing a trigonometric interpolant as $V(s)$ is not the only option. Indeed, any other function $v(x)$ whose half Laplacian is known and such that $v(\infty) = u(\infty)$ and $v(-\infty) = u(-\infty)$ can be used in principle. In this regard, we have considered $v(x) = (\pi/2)\erf(x)$, whose half Laplacian is given by \eqref{e:delta12erf} times $\pi/2$, where the constant $\pi/2$ is necessary in order to make $v(x)$ match $u(x)$ as $x\to\pm\infty$. In that case, $V(s) = (\pi/2)\erf(L\cot(s))$ satisfies $V^{(m)}(0) = V^{(m)}(\pi) = 0$, for all $m\in\mathbb N$, so $W(s) = U(s) - V(s)\in\mathbb C^{\infty}(\mathbb R)$.
	
	On the right-hand side of Figure \ref{f:fraclapatanVerf}, we show the corresponding errors, taking again\break$L\in\{0.01, 0.02, \ldots, 10\}$ and $N \in\{2^2, 2^3, \ldots, 2^{13}\}$. The results improve those on the left-hand side of Figure  \ref{f:fraclapatan} corresponding to an even extension of $U(s)$, although they are in general slightly worse than those on the left-hand side of Figure \ref{f:fraclapatanVerf}, except when $N = 64$ and $N = 128$. 
	
	In the results in Figure \ref{f:fraclapatanVerf}, we have chosen $V(s)$ in such a way that $W(s)$ has a high global regularity. Reciprocally, we might expect that examples where $W(s)$ is less regular give us worse results. To illustrate this, we have taken $u(x) = x(1 + x^2)^{-1/2}$, so $U(s) = L\cot(s)(1+L^2\cot^2(s))^{-1/2}$, and $(-\Delta)^{1/2}u(x)$ is given by \eqref{e:fraclapxsqrt1x2}. On the left-hand side of Figure \ref{f:fraclapxsqrt1x2V}, we plot the errors corresponding to $v(x) = (2/\pi)\arctan(x)$ (where the constant $2/\pi$ is necessary in order to make $v(x)$ match $u(x)$ as $x\to\pm\infty$), and, on the right-hand side of Figure \ref{f:fraclapxsqrt1x2V}, those corresponding to $v(x) = \erf(x)$. In both cases, we have taken $L\in\{0.01, 0.02, \ldots, 100\}$ and $N \in\{2^2, 2^3, \ldots, 2^{13}\}$, and the results of both graphics are very similar, and much worse than those of Figure \ref{f:fraclapxsqrt1x2} corresponding to the even extension of $U(s)$ at $s = \pi$. This is due to the fact that, when $V(s) = (2/\pi)\arctan(L\cot(s))$ or $V(s) = \erf(L\cot(s))$, then $W(s) = U(s) - V(s)\in\mathcal C^0(\mathbb R)$.
	
	\begin{figure}[!htbp]
		\centering
		\includegraphics[width=0.5\textwidth, clip=true]{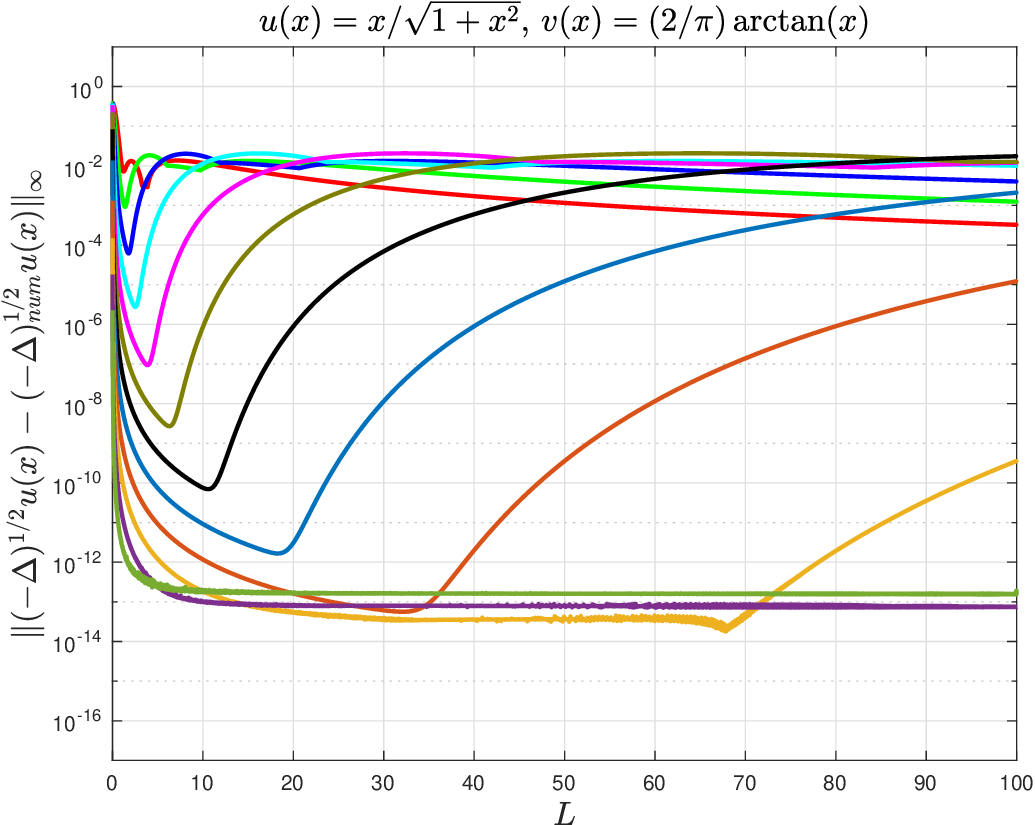}\includegraphics[width=0.5\textwidth, clip=true]{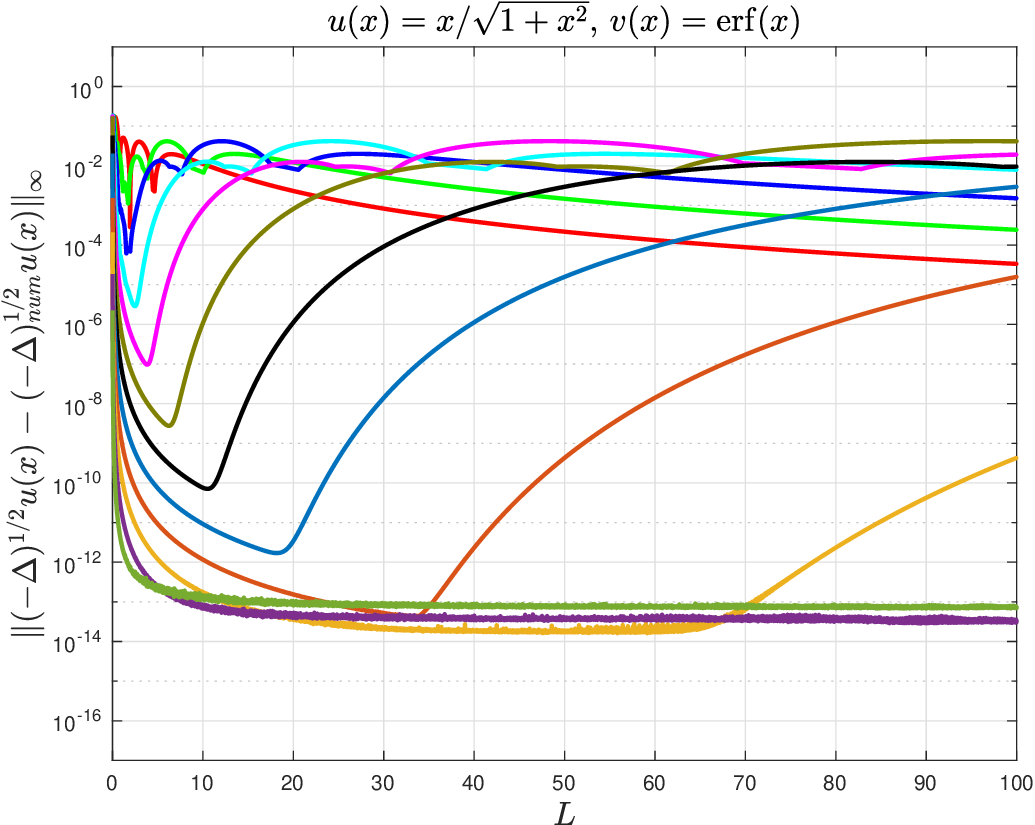}
		\includegraphics[width=\textwidth, clip=true]{legendnumberofpointsN.eps}
		\caption{Errors in the numerical approximation of $(-\Delta)^{1/2}u(x)$, for $u(x) = \arctan(x)$, taking $L\in\{0.01, 0.02, \ldots, 100\}$ and $N \in\{2^2, 2^3, \ldots, 2^{13}\}$. Left: we have taken $v(x) = (2/\pi)\arctan(x)$. Right: we have taken $v(x) = \erf(x)$.}
		\label{f:fraclapxsqrt1x2V}
	\end{figure}
	
	\subsection{Improving the periodic case}
	
	The procedure explained in Section \ref{s:nonperiodictoperiodic} can be applied not only when $U(s)$ is not periodic on $s\in[0,\pi]$, but also when it is, by choosing a periodic function $V(s)$ on $s\in[0,\pi]$ whose fractional Laplacian is explicitly known, and such that $W(s) = U(s) - V(s)$ is globally more regular on $s\in\mathbb R$. In order to illustrate this, we have considered again the example $u(x) = (1+x^2)^{-1/2}$ in Section \ref{s:numericalperiodic}. In that case, $U(s) = (1 + L^2\cot^2(s))^{-1/2}$, and we need to find $V(s)$ such that
	\begin{equation}
		\label{e:triginterp2}
		\begin{aligned}
			V(0) = U(0) & = 0, & V(\pi) & \equiv U(\pi) = 0,
			\cr
			V'(0) = U'(0) & = \frac{1}{L}, & V'(\pi) & \equiv U'(\pi) = -\frac{1}{L},
			\cr
			V'''(0)= U''(0) & = 0, & V''(\pi) & \equiv U''(\pi) = 0,
			\cr
			V'''(0) = U'''(0) & = \frac2L - \frac3{L^3}, & V'''(\pi) & \equiv U'''(\pi) = -\frac2L + \frac3{L^3},
			\cr
			V^{(4)}(0) = U^{(4)}(0) & = 0, & V^{(4)}(\pi) & \equiv U^{(4)}(\pi) = 0.
		\end{aligned}
	\end{equation}
	After proposing the ansatz
	$$
	V(s) = \sum_{k = 1}^{5}[\alpha_k\cos(ks) + \beta_k\sin(ks)],
	$$
	it is not difficult to conclude that the trigonometric interpolant that satisfies \eqref{e:triginterp2} is
	$$
	V(s) = \left(\frac{11}{8L} - \frac{3}{8L^3}\right)\sin(s) + \left(-\frac{1}{8L} + \frac{1}{8L^3}\right)\sin(3s),
	$$
	whose half Laplacian is explicitly known. Indeed, bearing in mind \eqref{e:deltas12sinks} and applying elementary trigonometric identities, we conclude that
	\begin{align*}
		(-\Delta)_s^{1/2}V(s) & = \frac1{16\pi}\bigg[\left(-\frac{14}{L^2} + \frac{6}{L^4}\right)\cos(s)  + \left(\frac{17}{L^2} - \frac{9}{L^4}\right)\cos(3s)
		\cr
		& \qquad \qquad + \left(-\frac{3}{L^2} + \frac{3}{L^4}\right)\cos(5s)\bigg]\ln\left(\cot\left(\frac s2\right)\right) + \frac1\pi\left(\frac{13}{8L^2} - \frac{5}{8L^4}\right) 
		\cr
		& \qquad + \frac1\pi\left(-\frac{2}{L^2} + \frac{1}{L^4}\right)\cos(2s) + \frac1\pi\left(\frac{3}{8L^2} - \frac{3}{8L^4}\right)\cos(4s).
	\end{align*}
	Then, it is straightforward to check that $W(s) = U(s) - V(s)\in\mathcal C^4(\mathbb R)$.
	
	\begin{figure}[!htbp]
		\centering
		\includegraphics[width=0.5\textwidth, clip=true]{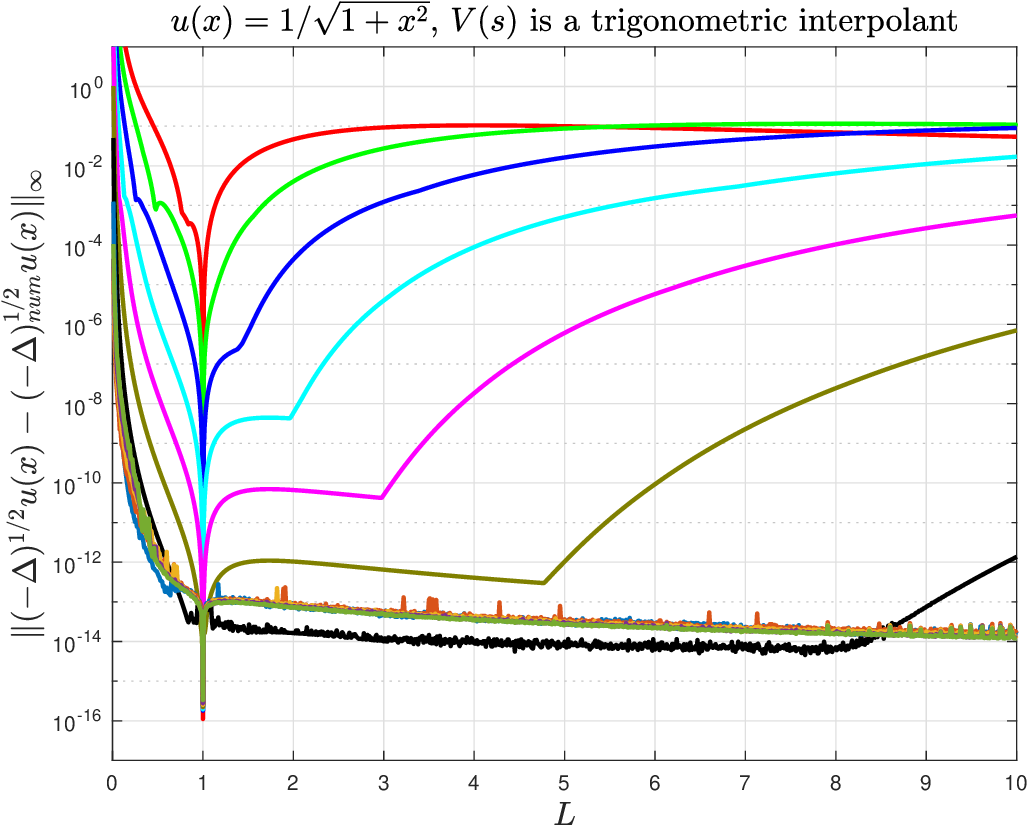}
		\includegraphics[width=\textwidth, clip=true]{legendnumberofpointsN.eps}
		\caption{Errors in the numerical approximation of $(-\Delta)^{1/2}u(x)$, for $u(x) = \arctan(x)$, taking $L\in\{0.01, 0.02, \ldots, 10\}$ and $N \in\{2^2, 2^3, \ldots, 2^{13}\}$. We have considered an even extension of $U(s)$ at $s = \pi$.}
		\label{f:fraclap1sqrt1x2V}
	\end{figure}
	
	In Figure \ref{f:fraclap1sqrt1x2V}, we have plotted the numerical errors for $L\in\{0.01, 0.02, \ldots, 10\}$ and $N \in\{2^2, 2^3, \ldots, 2^{13}\}$. When $L = 1$, $U(s) \equiv V(s)$, so there are only infinitesimally small rounding errors, and, in general, when $L \not = 1$, a high accuracy can be achieved from $N = 128$ on. In comparison with Figure \ref{f:fraclap1sqrt1x2}, the results are much better than those corresponding to the even extension, but not so good as those corresponding to the odd extension, which is still superior.
	
	\subsection{Discussion on the numerical experiments}
	
	When $U(s) \equiv u(x(s))$ is periodic of period $\pi$, the accuracy of the method depends on how well the truncated Fourier series \eqref{e:U(s)} approximates $U(s)$, and this is related to the asymptotic behavior of $u(x)$ as $x\to\pm\infty$, which determines the decay of the Fourier coefficients of $U(s)$. For instance, if $u(x)$ tends to a constant with a correction of $\mathcal O(1 / x)$ as $x\to\pm\infty$, then, $U(s)\in\mathcal C^0(\mathbb R)$ globally, and the Fourier coefficients of $U(s)$ decay quadratically, so the errors decay quadratically too, and larger values of $N$ and $L$ are in general required; however, a well-suited extension to $s\in[0,2\pi]$ can improve largely the results (see, e.g., Figure \ref{f:fraclap1sqrt1x2}, where an odd extension improves the results). On the other hand, if $u(x)$ has a faster decay, it is safe to work with $s\in[0,\pi]$ (see, e.g., Figure \ref{f:fraclap11x4}), and the implementation is simpler, as there is no need to extend the function to $[0,2\pi]$ to get good results.
	
	When $U(s)$ is not periodic of period $\pi$, i.e., when $\lim_{x\to-\infty}u(x)\not=\lim_{x\to\infty}u(x)$, then, in general, an even extension at $s = \pi$ can always be used (see, e.g., Figures \ref{f:fraclapxsqrt1x2} and \ref{f:fraclaperf}). However, when the decay rate of the correction as $x\to\pm\infty$ is of  $\mathcal O(1 / x)$, then larger values of $N$ and $L$ will be required, and it might be convenient to consider other smoother extensions (see, e.g., Figure \ref{f:fraclapatan}, where we have considered a trigonometric interpolant that gives $U(s)\in\mathcal C^4(\mathbb R)$).
	
	On the other hand, another option is to find a function $v(x)$ whose half Laplacian is explicitly known, and such that $w(x) = u(x) - v(x)$ satisfies $\lim_{x\to-\infty}w(x) = \lim_{x\to\infty}w(x)$, and $W(s)\equiv w(x(s))$ is periodic of period $\pi$ and globally more regular than $u(x)$. A possible option is to take $V(s) \equiv v(x(s))$ to be a trigonometric interpolant of $U(s)$; then, when $k$ is odd, $\cos(ks)$ and $\sin(ks)$ can be computed respectively by means of \eqref{e:deltas12cosks} and \eqref{e:deltas12sinks}, and the numerical computation of the half Laplacian is reduced to the case of period $\pi$.
	
	Summarizing, if the analytical expression of $u(x)$ is explicitly known, an even extension of $U(s)$ at $s = \pi$ is always possible, and particularly advisable if $u(x)$ decays fast at $x\to\infty$, although there are other possibilities that can be preferable if the global regularity of $U(s)$ is not too high, and this is specially true if an even extension of $U(s)$ at $s = \pi$ yields only a continuous function $U(s)\in\mathcal C^0(\mathbb R)$, because $u(x)$ tends to two respective constants, as $x\to\infty$, with a correction of $\mathcal O(1/x)$. On the other hand, in our opinion, the use of an even extension is not only completely justified, but also advisable if $u(x)$ is not explicitly known, and only approximations of $u(x)$ at the nodes $x = x_j$ are available, because we might not have enough accurate information to construct an interpolant $V(s)$ or a smooth extension at $s\in[\pi,2\pi]$. This is well illustrated in Section \ref{s:fisher} in the simulation of evolution equations like \eqref{e:fisher}, which have, additionally, a very different behavior when $x\to-\infty$ and when $x\to\infty$.
	
	\section{Simulation of a fractional Fisher's equation}
	
	\label{s:fisher}
	
	In order to illustrate the adequacy of the method in this paper, we have simulated the following initial-value problems:
	\begin{equation}
		\label{e:fisher1}
		\left\{
		\begin{aligned}
			& u_t = -(-\Delta)^{1/2}u + f(u), \quad x\in\mathbb R,\ t \ge 0,
			\cr
			& u(x, 0) = u_0(x) = \frac{1}{2} - \frac{x}{2\sqrt{1 + x^2}},
		\end{aligned}
		\right.
	\end{equation}
	and
	\begin{equation}
		\label{e:fisher2}
		\left\{
		\begin{aligned}
			& u_t = -(-\Delta)^{1/2}u + f(u), \quad x\in\mathbb R,\ t \ge 0,
			\cr
			& u(x, 0) = u_0(x) = 10^{-4}\left(\frac{1}{2} - \frac{x}{2\sqrt{1 + x^2}}\right),
		\end{aligned}
		\right.
	\end{equation}
	for $f(u) = u(1-u)$, which correspond to the fractional Fisher's equation \eqref{e:fisher} taking as initial data \eqref{e:u0} and \eqref{e:u0b}, respectively.
	
	We have used a classical fourth-order Runge-Kutta method to advance in time, taking $N = 1048576$, $L = 500000$, $\Delta t  = 0.1$. Note that, in evolution problems on an infinite domain like this one, where there is no domain truncation, the boundary conditions need not be explicitly enforced (see, e.g., \cite{Boyd1987}); i.e., we do not need to introduce explicitly in the codes the behaviour of $u(x, t)$ and/or its derivatives, as $x\to\pm\infty$. In this regard, it is especially interesting the discussion in \cite{Boyd1987} on the differences between natural boundary conditions, which do not require modifying the expansion functions (as in our case), and essential boundary conditions, which must be enforced individually on each basis function; indeed, when a boundary condition is natural, the differential equation itself obliges the boundary condition to be satisfied by the sum of the basis functions.
	
	In order to implement it in Matlab, we transform Listing \ref{code:nonperiodic} into a function, namely\break\verb"function_half_laplacian", which can be done immediately, by adding at the beginning of the code the lines
	\begin{verbatim}
		function Fu_num=function_half_laplacian(u2,L)
		N=length(u2)/2; % number of nodes
	\end{verbatim}
	and removing the assignments of \verb"N", \verb"L", \verb"u" and \verb"u2". Then, a minimal implementation in Matlab of, e.g., \eqref{e:fisher1} takes just a few lines of code.
	
\begin{lstlisting}[language=Matlab, basicstyle={\footnotesize\ttfamily}, caption = {Simulation of \eqref{e:fisher1}}]
% Simulate u_t=-(-\Delta)^(1/2)u+f(u), with f(u)=u(1 - u)
clear
f=@(u)u.*(1-u);
N=1048576/8; % number of nodes
L=500000;% scale factor
tmax=26; % final time
dt=.1; % \Delta t
mmax=round(tmax/dt); % number of time steps
s=pi*(.5+(0:N-1))/N; % s_j\in(0,\pi)
x=L*cot(s); % x_j\in(-\infty,infty)
u0=.5-.5*x./sqrt(1+x.^2); % initial data
u=u0;
for m=1:mmax
    k1_u=-function_half_laplacian([u,u(N:-1:1)],L)+f(u);
    uaux=u+.5*dt*k1_u;
    k2_u=-function_half_laplacian([uaux,uaux(N:-1:1)],L)+f(uaux);
    uaux=u+.5*dt*k2_u;
    k3_u=-function_half_laplacian([uaux,uaux(N:-1:1)],L)+f(uaux);
    uaux=u+dt*k3_u;
    k4_u=-function_half_laplacian([uaux,uaux(N:-1:1)],L)+f(uaux);
    u=u+dt*(k1_u+2*k2_u+2*k3_u+k4_u)/6;
end
\end{lstlisting}
	
	As can be seen, we have considered an even extension also in the intermediate steps of each iteration, which yields a very stable implementation, even for rather large value of $\Delta t$. On the other hand, the large values of $N$ and $L$ are necessary, in order to capture correctly the accelerating moving front solutions.
	
	\begin{figure}[!htbp]
		\centering
		\includegraphics[width=0.5\textwidth, clip=true]{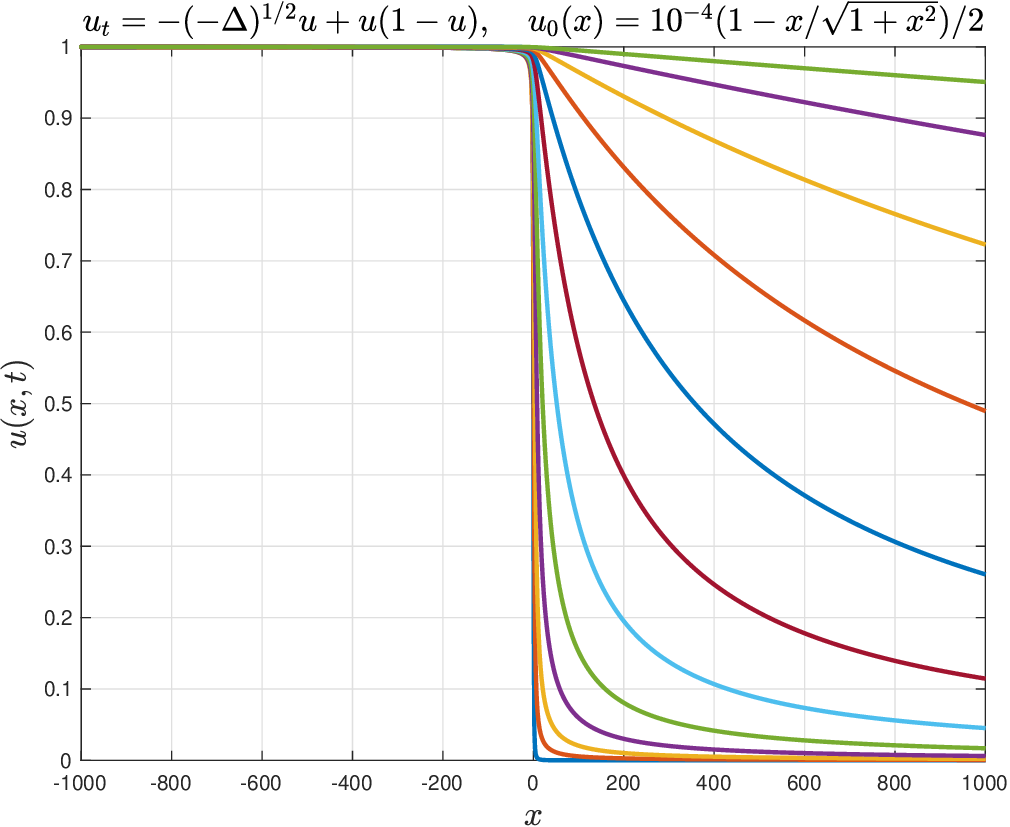}\includegraphics[width=0.5\textwidth, clip=true]{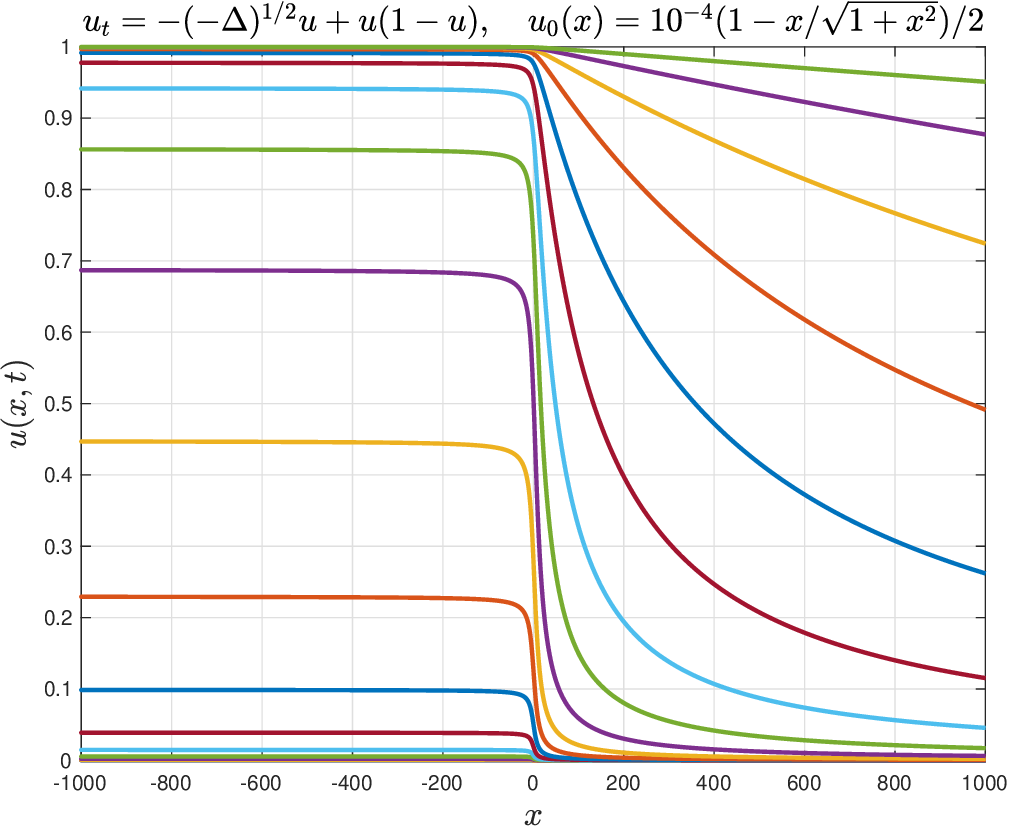}
		\caption{Left: plot of $u(x, t)$ corresponding to \eqref{e:fisher1}, at $t\in\{0, 1, \ldots, 11\}$. Right: plot of $u(x, t)$ corresponding to \eqref{e:fisher2}, at $t\in\{0, 1, \ldots, 18\}$.}
		\label{f:fisherinitial}
	\end{figure}
	
	\begin{figure}[!htbp]
		\centering
		\includegraphics[width=0.5\textwidth, clip=true]{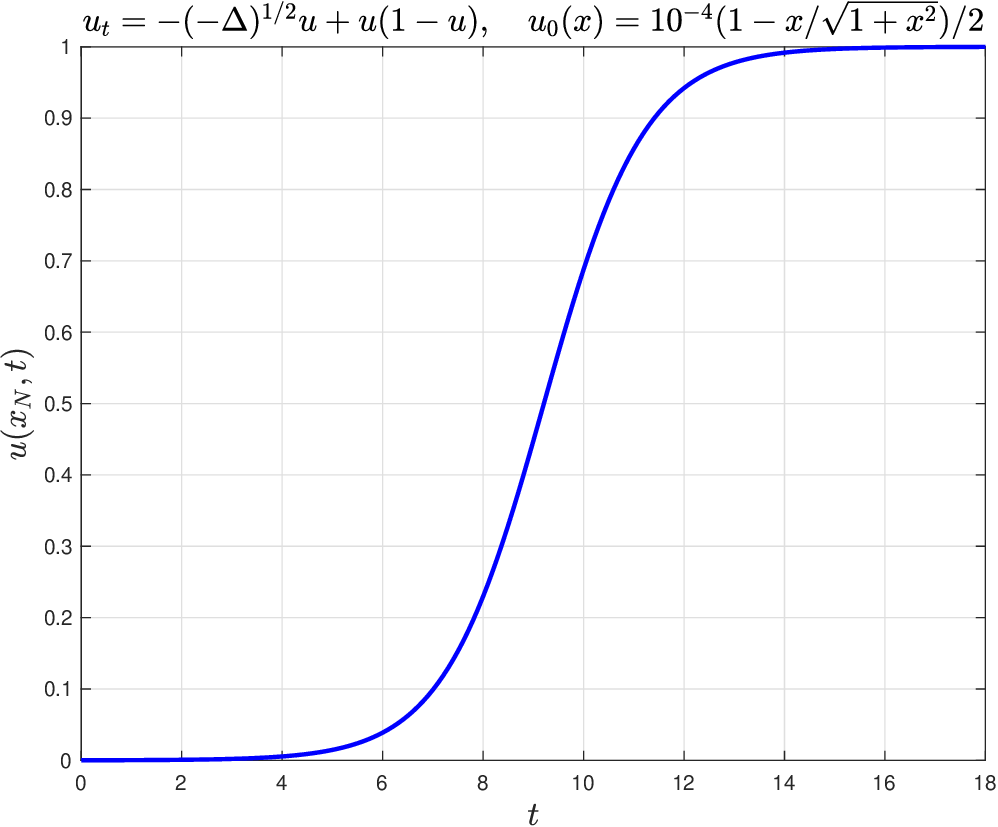}\includegraphics[width=0.5\textwidth, clip=true]{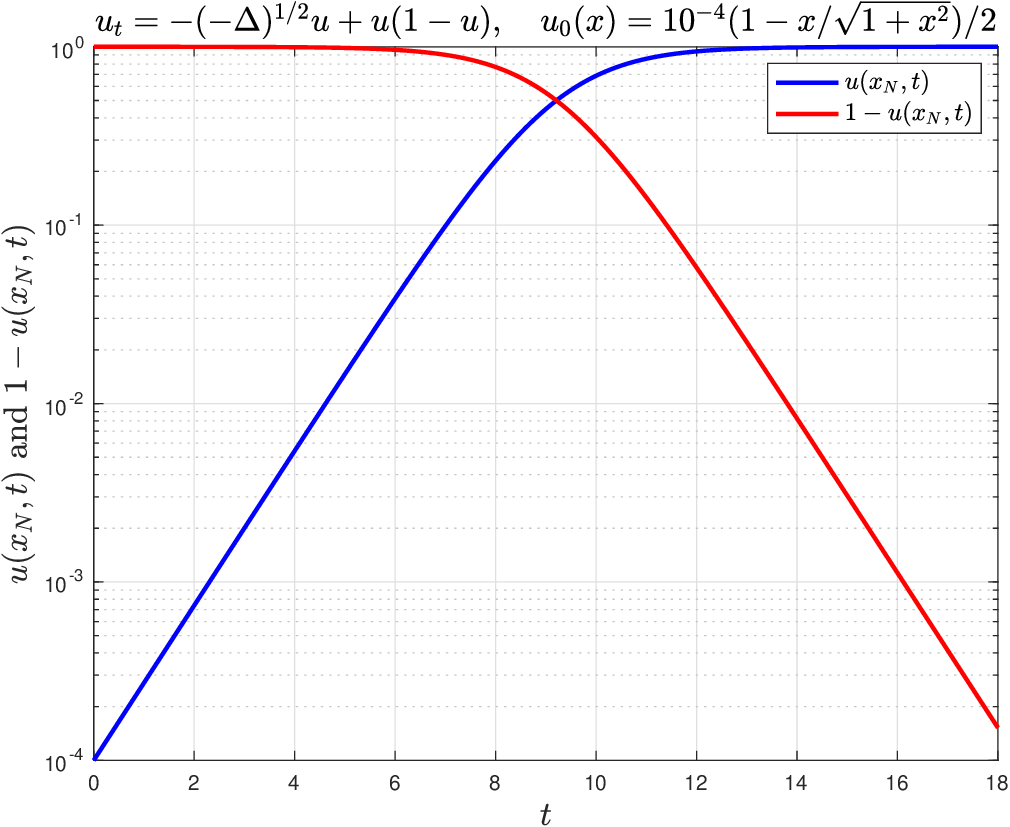}
		\caption{Left: plot of $u(x, t)$ corresponding to \eqref{e:fisher1}, at $t\in\{0, 1, \ldots, 18\}$. Right: plot of $u(x, t)$ corresponding to \eqref{e:fisher2}, at $t\in\{0, 1, \ldots, 18\}$.}
		\label{f:evolxminusinf}
	\end{figure}
	
	In Figure \ref{f:fisherinitial}, we have plotted, on the left-hand side, $u(x, t)$ corresponding to \eqref{e:fisher1} at $t\in\{0, 1, \ldots, 11\}$; and on the right-hand side, $u(x, t)$ corresponding to \eqref{e:fisher1} at $t\in\{0, 1, \ldots, 18\}$. Observe that in \eqref{e:fisher1}, $u_0(x)\to1$, as $x\to-\infty$, so $u(x, t)\to1$, as $x\to-\infty$, for all $t > 0$. However, in \eqref{e:fisher2}, $u_0(x)\to10^{-4}$, as $x\to-\infty$, so $u(x, t)$, for $x \ll -1$, tends to the stable state $u = 1$, as $t$ increases. Once that $u(x, t)$, for $x \ll -1$, is close enough to $u = 1$, the evolution of \eqref{e:fisher1} and that of \eqref{e:fisher2} are similar, but the latter has a delay of approximately $7$ time unities with respect to the former. In this regard, it is interesting to analyze how $u(x, t)$ evolves from $u = 10^{-4}$ to $u = 1$, as $x\to-\infty$. To clarify this, we have plotted, on the left-hand side of Figure \ref{f:evolxminusinf}, $u(x_{N-1}, t)$ at $t\in\{0, 0.1, \ldots, 18\}$, where $x_{N-1} = L\cot(s_{N-1}) = L \cot(\pi (2N - 1) / (2N)) = -3.3377\times10^{11}$ is the most negative spatial node. Moreover, on the right-hand side of Figure \ref{f:evolxminusinf}, we have plotted in semilogarithmic scale both $u(x_{N-1}, t)$ and $1 - u(x_{N-1}, t)$ at $t\in\{0, 0.1, \ldots, 18\}$, which suggests that near $u = 10^{-4}$ and near $u = 1$, respectively, the curves behave approximately as straight lines, which suggests at its turn an exponential behavior of $u(x, t)$ as $x\to-\infty$, when $u$ leaves $u = 10^{-4}$, and also when $u$ approaches $1$. A deeper study of this interesting fact lies beyond the scope of this paper, but serves very well to illustrate the adequacy of an even extension of $U(s, t)$ at $s = \pi$, whereas another kind of extension $U(s, t)$ at $s\in[\pi,2\pi]$, or finding some $V(s, t)$ such that $U(s, t) - V(s, t)$ is periodic of period $\pi$ and more regular, would greatly increase the complexity of the code, and it is not completely clear, how to do it in an ideal manner. Here, it must underlined that $u(x, t)$ tends to a constant exponentially, as $x\to-\infty$, but has slow, algebraic decay, as $x\to\infty$, and the even extension captures correctly this behavior.
	
	Besides being able to deal with situations where the behavior of $u(x, t)$ is a priori not known, as $x\to-\infty$ or as $x\to\infty$, we are also capable of capturing the accelerating fronts that ensue, even when they are found at very large values of $x$. On the left-hand side of Figure \ref{f:fisher}, we plot the numerical approximation of $u(x, t)$ corresponding to \eqref{e:fisher1} at $t \in\{4, 5, \ldots, 25\}$; and on the right-hand side of Figure \ref{f:fisher}, we plot the numerical approximation of $u(x, t)$ corresponding to \eqref{e:fisher2} at $t \in\{11, 12, \ldots, 32\}$; in both cases, the abscissa axis is in logarithmic scale.
	
	In our opinion, it is remarkable how the numerical method allows us to capture the fronts, even if it is located at huge values of $x$ (i.e., of the order of $\mathcal O(10^{10}))$. We remark that, as $t$ increases, the curves are approximately parallel and equidistant, suggesting that the speed grows exponentially fast in time. Furthermore, as noted in Figure \ref{f:fisherinitial}, the evolution of \eqref{e:fisher1} and that of \eqref{e:fisher2} are similar, but the latter has a delay of approximately $7$ time unities
	\begin{figure}[!htbp]
		\centering
		\includegraphics[width=0.5\textwidth, clip=true]{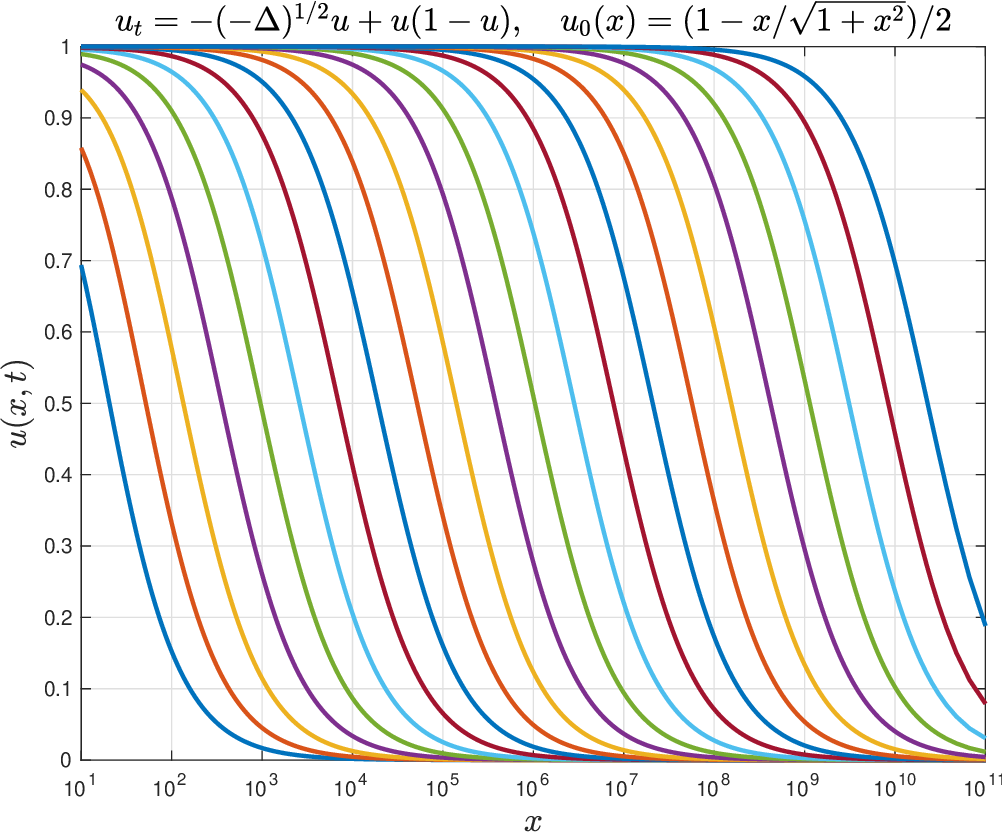}\includegraphics[width=0.5\textwidth, clip=true]{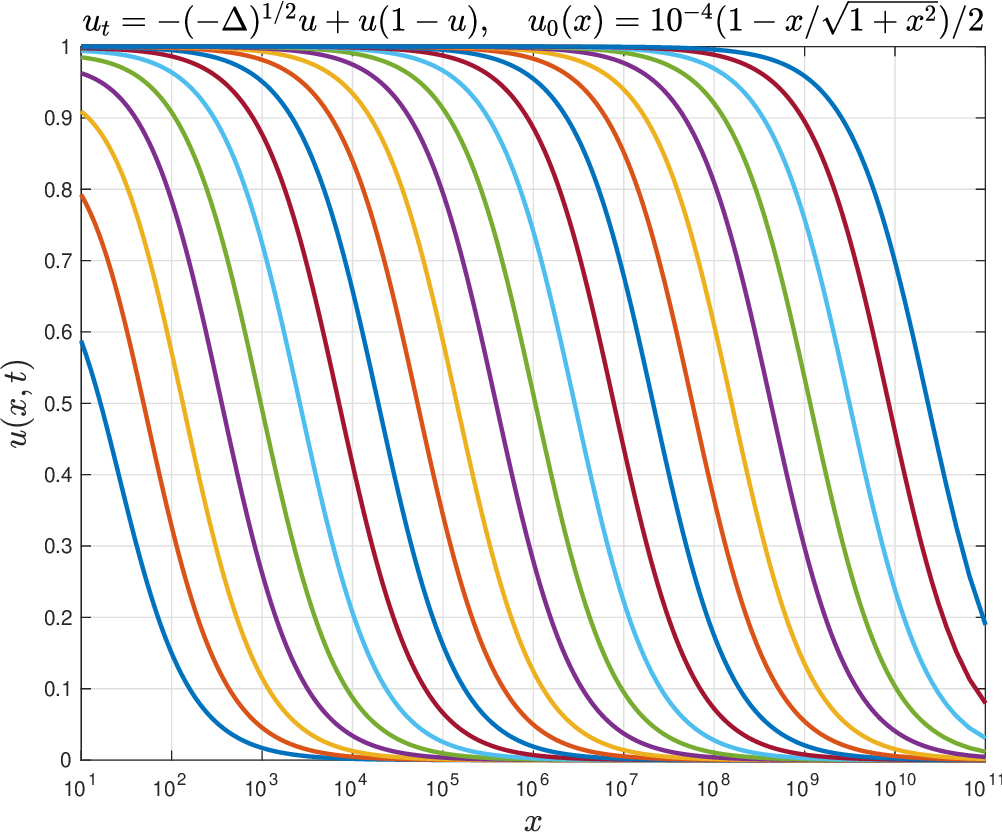}
		\caption{Left: plot of $u(x, t)$ corresponding to \eqref{e:fisher1}, at $t\in\{4, 5, \ldots, 25\}$. Right: plot of $u(x, t)$ corresponding to \eqref{e:fisher2}, at $t\in\{11, 12, \ldots, 32\}$. In both cases, the abscissa axis is in logarithmic scale.}
		\label{f:fisher}
	\end{figure}
	
	In order to measure numerically the traveling speed, we have approximated numerically, as in \cite{cayamacuestadelahoz2021}, $x_{0.5}(t)$, i.e., the value of $x$, such that $u(x, t) = 0.5$, by using a bisection algorithm together with spectral interpolation. On the left-hand side of Figure \ref{f:fisherx05}, we plot $\ln(x_{0.5}(t))$ corresponding to \eqref{e:fisher1}, at $t\in\{0.1, 0.2, \ldots, 25\}$; and on the right-hand side of Figure \ref{f:fisherx05}, we plot $\ln(x_{0.5}(t))$ corresponding to \eqref{e:fisher2}, at $t\in\{10, 10.1, \ldots, 32\}$ (note that, in this case, for smaller times, $x_{0.5}(t) < 0$). In both cases, we have a curve that tends to a straight line with slope equal to $1$, as predicted in \cite{CabreRoquejoffre2013}. In fact, the least-square regression line of the points $(t,\ln(x_{0.5}(t)))$ corresponding to \eqref{e:fisher1}, at $t\in\{15, 15.1, \ldots, 25\}$, has a slope equal to $1 - 2.1470\times10^{-6}$, and its Pearson correlation coefficient is $r = 1 - 7.4519\times10^{-10}$; and the least-square regression line of the points $(t,\ln(x_{0.5}(t)))$ corresponding to \eqref{e:fisher2}, at $t\in\{23, 23.1, \ldots, 32\}$, has a slope equal to $1 - 2.9403\times10^{-6}$, and its Pearson correlation coefficient is $r = 1 - 8.2824\times10^{-10}$. Therefore, the numerical results fully confirm the theoretical results in \cite{CabreRoquejoffre2013}, and show that an even extension of $U(s)$ at $s = \pi$ is adequate to deal efficiently with demanding problems, like the numerical simulation of \eqref{e:fisher1} and \eqref{e:fisher2}.
	
	\begin{figure}[!htbp]
		\centering
		\includegraphics[width=0.5\textwidth, clip=true]{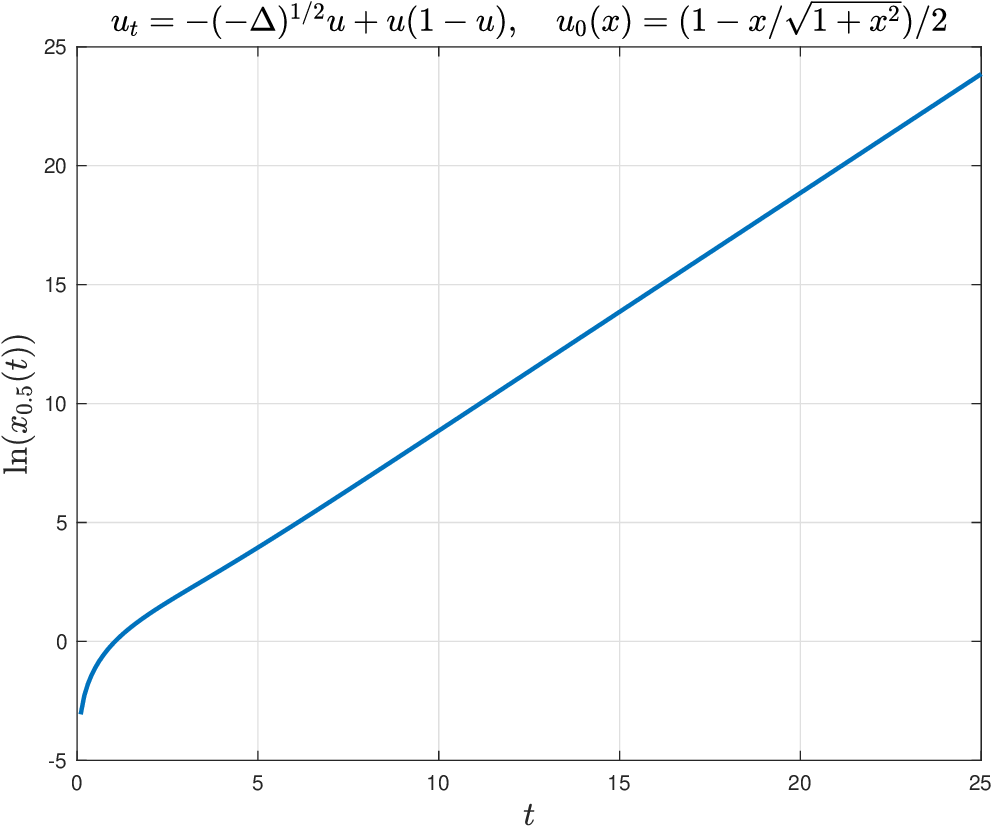}\includegraphics[width=0.5\textwidth, clip=true]{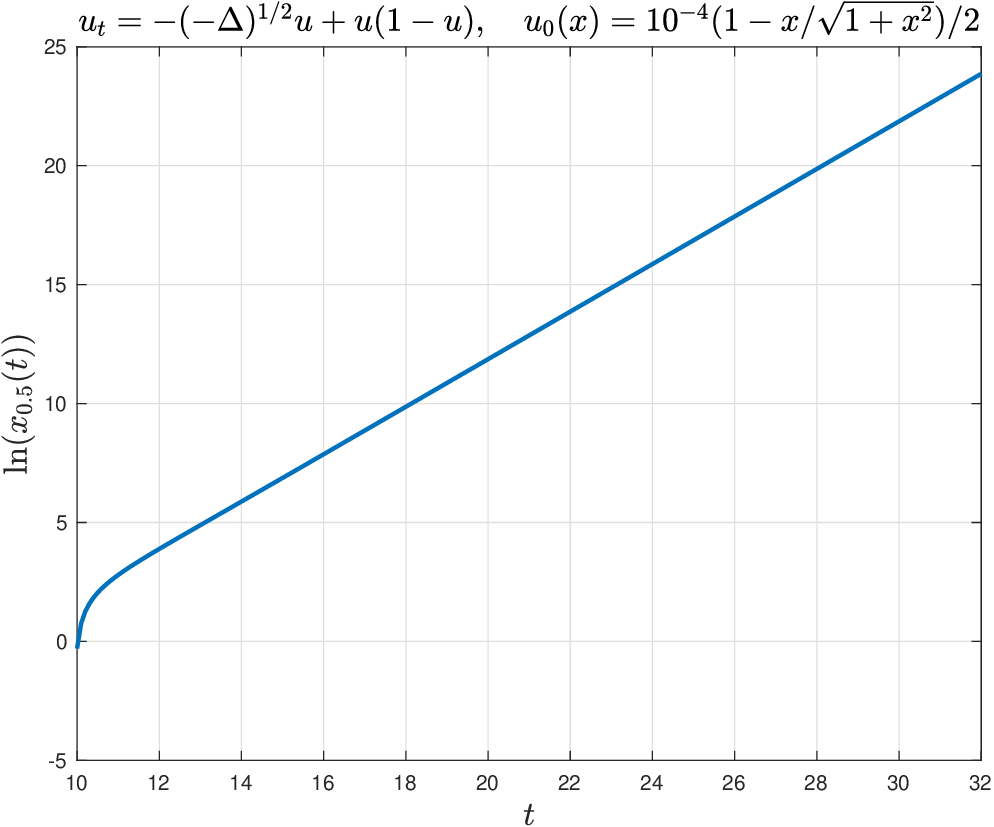}
		\caption{Left: plot of $\ln(x_{0.5}(t))$ corresponding to \eqref{e:fisher1}, at $t\in\{0.1, 0.2, \ldots, 25\}$. Right: plot of $\ln(x_{0.5}(t))$ corresponding to \eqref{e:fisher2}, at $t\in\{10, 10.1, \ldots, 32\}$.}
		\label{f:fisherx05}
	\end{figure}
	
	\section*{Acknowledgments} We want to thank the anonymous referee for his/her valuable comments, that have helped to improve greatly the quality of this paper.
	
	\appendix
	
	\section{Calculation of $(-\Delta)^{1/2}(1+x^4)^{-1}$}
	\label{s:appendix}
	The aim of this appendix is to show how to compute \eqref{e:fraclap11x4} by hand, by using complex variable techniques. Note that the same procedure can be applied in order to deduce the half Laplacian of many other definitions of $u(x)$.
	
	We use the expression for the half Laplacian given by \eqref{e:fraclaplmathematica}. Therefore, given $u(x) = (1+x^4)^{-1}$, the integrand of \eqref{e:fraclaplmathematica} is given by
	\begin{equation*}
		\frac{u_x(x-y) - u_x(x+y)}{y} = \frac{4(x+y)^3(1+(x-y)^4)^2 - 4(x-y)^3(1+(x+y)^4)^2}{y(1+(x+y)^4)^2(1+(x-y)^4)^2}.
	\end{equation*}
	This suggests defining
	$$
	g(z; x) = \frac{4(x+z)^3(1+(x-z)^4)^2 - 4(x-z)^3(1+(x+z)^4)^2}{z(1+(x+z)^4)^2(1+(x-z)^4)^2},
	$$
	and taking the integration contour $C = C_1\cup C_r\cup C_2\cup C_R$, with
	\begin{align*}
		C_1 & = \{y\ |\ y\in(-R, -r)\},
		\cr
		C_r & = \{re^{-i\theta}\ |\ \theta\in(-\pi, 0)\},
		\cr
		C_2 & = \{y\ |\ y\in(r, R)\},
		\cr
		C_R & = \{Re^{i\theta}\ |\ \theta\in(0, \pi)\},
	\end{align*}
	where $R, r > 0$. Then, choosing $R$ large enough and $r$ small enough, we have
	\begin{align*}
		\int_Cg(z; x)dz & = \int_{C_1}g(z; x)dz + \int_{C_r}g(z; x)dz + \int_{C_2}g(z; x)dz + \int_{C_R}g(z; x)dz
		\cr
		& = 2\pi i\Res [g(z; x), -x + e^{i\pi/4}] + 2\pi i\Res [g(z; x), -x + e^{i3\pi/4}]
		\cr
		& \qquad + 2\pi i\Res [g(z; x), x + e^{i\pi/4}] + 2\pi i\Res [g(z; x), x + e^{i3\pi/4}],
	\end{align*}
	and this expression is true, when we make $r\to0^+$ and $R\to\infty$.
	
	On the one hand,
	$$
	\lim_{\substack{r\to0^+\\R\to\infty}}\int_Cg(z; x)dz = 2\int_{0}^{\infty}g(y; x)dy,
	$$
	because
	\begin{align*}
		\lim_{\substack{r\to0^+\\R\to\infty}}\int_{C_1}g(z; x)dz & = \int_{-\infty}^{0}g(y; x)dy = \int_{0}^{\infty}g(y; x)dy,
		\cr
		\lim_{\substack{r\to0^+\\R\to\infty}}\int_{C_2}g(z; x)dz & = \int_{0}^{\infty}g(y; x)dy,
		\cr
		\lim_{r\to0^+}\int_{C_r}g(z; x)dz & = -\pi i\Res[g(z; x), 0] = -\pi i[g(z; x)z]\bigg|_{z=0} = 0,
		\cr
		\lim_{R\to\infty}\int_{C_R}g(z; x)dz & = iR\lim_{R\to\infty}\int_{0}^{\pi}g(Re^{i\theta}; x)e^{i\theta}d\theta = 0,
	\end{align*}
	where we have used in the first line that $g(-y, x) = g(y, x)$; and in the fourth line, that, fixed $x$, the numerator of $g(z; x)$ behaves as $\mathcal O(R^{11})$ and the denominator behaves as $\mathcal O(R^{17})$, as $R\to\infty$.
	
	On the other hand,
	\begin{align*}
		\Res [g(z; x), -x + e^{i\pi/4}] & = \frac{d}{dz}[g(z; x)(z - (-x + e^{i\pi/4}))^2]\bigg|_{z = -x + e^{i\pi/4}}
		\cr
		& = \frac{{{(x^2 +i)}}^2 {({(1+i)}\sqrt{2}x^2 +4ix+\sqrt{2}{(-1+i)})}}{8{{(1 + x^4)}}^2 },
		\cr
		\Res [g(z; x), -x + e^{3i\pi/4}] & = \frac{d}{dz}[g(z; x)(z - (-x + e^{3i\pi/4}))^2]\bigg|_{z = -x + e^{3i\pi/4}}
		\cr
		& = \frac{{{(x^2 -i)}}^2 {({(-1+i)}\sqrt{2}x^2 -4ix+\sqrt{2}{(1+i)})}}{8{{(1+x^4)}}^2 },
		\cr
		\Res [g(z; x), x + e^{i\pi/4}] & = \frac{d}{dz}[g(z; x)(z - (x + e^{i\pi/4}))^2]\bigg|_{z = x + e^{i\pi/4}}
		\cr
		& = \frac{{{(x^2 +i)}}^2 {({(1+i)}\sqrt{2}x^2 - 4xi+\sqrt{2}{(-1+i)})}}{8{{(1+x^4)}}^2 },
		\cr
		\Res [g(z; x), x + e^{3i\pi/4}] & = \frac{d}{dz}[g(z; x)(z - (x + e^{3i\pi/4}))^2]\bigg|_{z = x + e^{3i\pi/4}}
		\cr
		& = \frac{{{(x^2 -i)}}^2{({(-1+i)}\sqrt{2}x^2 +4ix+\sqrt{2}{(1+i)})}}{8{{(1+x^4)}}^2 }.
	\end{align*}
	Bearing in mind all the previous arguments,
	\begin{align*}
		(-\Delta)^{1/2}\frac{1}{1+x^4} & = \frac1\pi\int_{0}^{\infty}g(y; x)dy
		\cr
		& =  i\Res [g(z; x), -x + e^{i\pi/4}] +  i\Res [g(z; x), -x + e^{i3\pi/4}]
		\cr
		& \qquad +  i\Res [g(z; x), x + e^{i\pi/4}] +  i\Res [g(z; x), x + e^{i3\pi/4}]
		\cr
		& = \frac{(1 - x^2)(1 + 4x^2 + x^4)}{\sqrt2{{(1+x^4)}}^2},
	\end{align*}
	which concludes the proof of \eqref{e:fraclap11x4}.

\end{document}